\documentclass[11pt]{article}

\usepackage{multirow}
\usepackage{algorithmic}
\usepackage{amsmath}
\usepackage{amsfonts}
\usepackage{amsthm}
\usepackage[english]{babel} 
\usepackage{dsfont}
\usepackage{amssymb}
\usepackage{graphicx}        % standard LaTeX graphics tool
\usepackage{cite}                            % when including figure files
\usepackage{subfigure}
\usepackage{epstopdf}
\usepackage{epsfig}
\usepackage[table]{xcolor}
\usepackage{subfloat}
\usepackage{float}
\usepackage[thinlines]{easytable}
\usepackage{array}
\usepackage{mathtools}
\usepackage{stmaryrd}
\usepackage{ulem}
\usepackage[final]{showlabels}%inline or final

\usepackage{booktabs}
\usepackage{url}
\usepackage[font=small,labelfont=bf,textfont=it,indention=.1cm,width=1.1\textwidth]{caption}

\usepackage[capitalise]{cleveref}

\allowdisplaybreaks
\numberwithin{equation}{section}

\newtheorem{remark}{Remark}[section]

\newtheorem{lemma}{Lemma}[section]
\newtheorem{theorem}{Theorem}[section]
\newtheorem{corollary}{Corollary}[section]
\newtheorem{proposition}{Proposition}[section]
\usepackage[linesnumbered,algoruled,boxed,lined]{algorithm2e}

\def\RR{\mathbb R}
\def\EE{\mathcal E}

\def\a{\alpha}
\def\b{\beta}

\def\diag{{\rm diag}}

\def\argmin{{\rm arg}\!\min}
\def\be{\begin{equation}}
\def\ee{\end{equation}}
\def\bea{\begin{eqnarray}}
\def\eea{\end{eqnarray}}

\newtheorem{assumption}{Assumption}[section]
\newcommand{\mc}[1]{\mathcal{#1}}

\newcommand{\la}{\langle}
\newcommand{\ra}{\rangle}
\newcommand{\lp}{\left(}
\newcommand{\rp}{\right)}
\newcommand{\diff}{{\rm d}}
\DeclareMathOperator{\Diag}{Diag}
\DeclareMathOperator{\amin}{argmin}
\title{Binary interaction methods for high dimensional global optimization and machine learning}

\author{Alessandro Benfenati\footnote{University of Milan, Department of Environmental Science and Policy (alessandro.benfenati@unimi.it)}
\and Giacomo Borghi\footnote{RWTH Aachen University, Department of Mathematics (borghi@eddy.rwth-aachen.de)}
\and Lorenzo Pareschi\footnote{University of Ferrara, Department of Mathematics and Computer Science (lorenzo.pareschi@unife.it)}}

\newcommand{\AB}[1]{\textcolor{black}{#1}}

\newcommand{\rev}[1]{\textcolor{black}{#1}}
\newcommand{\revlp}[1]{\textcolor{black}{#1}}

\newcommand{\revn}[1]{\textcolor{blue}{#1}}

\renewcommand{\emph}[1]{{\it #1}}
\renewcommand{\epsilon}{\varepsilon}

\begin{document}
\maketitle

\begin{abstract}
In this work we introduce a new class of gradient-free global optimization methods based on a binary interaction dynamics governed by a Boltzmann type equation. In each interaction the particles act taking into account both the best microscopic binary position and the best macroscopic collective position. \rev{For the resulting kinetic optimization methods, convergence to the global minimizer is guaranteed for a large class of functions under appropriate parameter constraints that do not depend on the dimension of the problem.} In the mean-field limit we show that the resulting Fokker-Planck partial differential equations generalize the current class of consensus based optimization (CBO) methods.  Algorithmic implementations inspired by the well-known direct simulation Monte Carlo methods in kinetic theory are derived and discussed. Several examples on prototype test functions for global optimization are reported including \rev{an application} to machine learning.
\end{abstract}

\tableofcontents

%\newpage
\section{Introduction}
A new class of numerical methods for global optimization based on particle dynamics has been introduced in some recent articles\cite{pinnau2017consensus,carrillo2018analytical,carrillo2019consensus,fhps20-1,fhps20-2,fhps20-3,mtw20,carrillo2021consensus}. These methods, referred to as consensus based optimization (CBO) methods for the similarities between the particle dynamics in the minimizer and consensus dynamics in opinion formation, fall within the large class of metaheuristic methods 
\cite{Aarts:1989:SAB:61990,Back:1997:HEC:548530,Blum:2003:MCO:937503.937505,Gendreau:2010:HM:1941310}. Among popular metaheuristic methods we recall the simplex heuristics \cite{NeldMead65}, evolutionary programming \cite{Fogel:2006:ECT:1202305},  the  Metropolis-Hastings sampling algorithm \cite{hastings70}, genetic algorithms \cite{Holland:1992:ANA:531075}, particle swarm optimization (PSO) \cite{kennedy2010particle,poli2007particle}, ant colony optimization (ACO) \cite{dorigo2005ant}, simulated annealing (SA) \cite{holley1988simulated,kirkpatrick1983optimization}. 

In contrast to classic metaheuristic methods, for which it is quite difficult to provide rigorous convergence to global minimizers (especially for those methods that combine instantaneous decisions with memory mechanisms), CBO methods, thanks to the instantaneous nature of the dynamics permit to exploit mean-field techniques to prove global convergence for a large class of optimization problems \cite{carrillo2018analytical,carrillo2019consensus,fhps20-2,fornasier2021consensusbased}. Despite their simplicity CBO methods seem to be powerful and robust enough to tackle many interesting high dimensional non-convex optimization problems of interest in machine learning \cite{carrillo2019consensus,fhps20-2,jinew20}. 

As shown in \cite{carrillo2019consensus,fhps20-2} in practical applications the methods benefit from the use of small batches of interacting particles since the global collective decision mechanism may otherwise lead the model to be more easily trapped in local minima. For these CBO methods based on small batches, however, a robust mathematical theory is still missing. We mention also that, recently, a continuous description of PSO methods based on a system of stochastic differential equations was proposed in \cite{gp20} and its connections with CBO methods analyzed through the corresponding mean-field descriptions. Rigorous results concerning the mean-field limit \revlp{of PSO methods and the corresponding CBO dynamics} have been subsequently presented in \cite{hui20}. \revlp{We refer the reader to the recent surveys \cite{grassi2021meanfield, totzeck2021trends} for a more complete overview.}

Motivated by this, in the present paper we \rev{introduce} a new class of kinetic theory based optimization (KBO) methods algorithmically solved by particle dynamics to address the following  optimization problem
\begin{equation}\label{typrob}
v^\star \in \argmin\limits_{v\in \RR}\EE(v)\,,
\end{equation}
where $\EE(v):\mathbb R^{d} \to \mathbb R$ is a given continuous cost functions, which we wish to minimize. \revn{In the following, we will assume that the minimizing argument $v^\star$ of \eqref{typrob} exists and is unique.}

Both statistical estimation and machine learning consider the problem of minimizing an objective function in the form of a sum
\begin{equation}\label{machlear}
\EE(v)=\frac1{n}\sum_{i=1}^n \EE_i(v),
\end{equation}
where each summand function $\EE_i$ is typically associated with the  $i$-observation in the data set, for example used for training \cite{bishop06}.
In statistics, the problems of minimizing the sum occur in least squares, in the estimation of the highest probability (for independent observations), and \rev{more general} in $M$-estimators \cite{fumio00}. The problem of sum minimization also arises for the minimization of empirical risk in statistical learning \cite{vapnik91}. In this case, $\EE_i$ is the value of the loss function at $i$-th example, and $\EE$ is the empirical risk.

In many cases, the summand functions have a simple form that enables inexpensive evaluations of the sum-function and the sum gradient. First order methods, such as (stochastic) {gradient descent methods}, are preferred both because of speed and scalability and because they are considered generically able to escape the trap of critical points. However, in other cases, evaluating the sum-gradient may require expensive evaluations of the gradients and/or some of the functions may be noisy or discontinuous. Additionally, most gradient-based optimizers are not designed to handle multi-modal problems or discrete and mixed discrete-continuous design variables. Gradient-free methods, such as the metaheuristics approaches mentioned before, may therefore represent a valid alternative. 

In contrast to previous CBO approaches, where the dynamic was of mean-field type, the new KBO methods are based on binary interactions between agents which can estimate the best position \rev{according} to a combination of a local interaction and a global alignment process. Binary interactions are inspired by similar processes of social alignment in kinetic models for opinion formation,  \revlp{where agents modify their opinions according to a process of local compromise with other agents and the global influence of external media \cite{albiandco2016opinion, albi2017mean, partos13,BENFENATI2013979,Benfenati14,HVP,APZN}.} The corresponding dynamic is therefore described by a multidimensional Boltzmann equation that is solved by adapting the well-known direct simulation Monte Carlo methods \cite{Bird, Nanbu, PRMC} to the present case. We emphasize that, the resulting schemes present some analogies with the recently introduced random batch methods in the case of small batches of size two \cite{AlPa, JLJ, KHJD}.

In particular, we show that, in a suitable scaling derived from the quasi-invariant limit in opinion dynamic, the corresponding mean-field dynamic is governed by CBO methods. Noticeably, the resulting CBO methods generalize the classical CBO approach in \cite{pinnau2017consensus,carrillo2019consensus} by preserving memory of the microscopic interaction dynamic. As shown by the numerical experiments, an interesting aspect in this direction is that the kinetic optimization model is able to capture the global minimum even in the case where there is no global alignment process, as in the original CBO models, but only a local alignment process where information is shared only between pairs of particles.  

The rest of the paper is organized as follows. In the next \rev{section,} we introduce the kinetic model and the corresponding Boltzmann equation. Section \ref{S3} is then devoted to analyze the main properties of the kinetic model and to consider a suitable scaling limit which permits to derive the analogous mean-field optimizers of CBO type. \rev{Convergence to the global optimum for KBO methods is then studied in Section 4, where we demonstrate exponentially fast convergence to the minimum, with a constraint on the parameters independent of the dimension for binary interactions with anisotropic noise. Finally, in Section \ref{S5}  we present several numerical experiments including an application to a machine learning problem}. Some concluding remarks are then given at the end of the manuscript.

\section{A kinetic model for global optimization}
\label{S2}
\revlp{In analogy to some key concepts of metaheuristic optimization methods based on particle dynamics, in the following we introduce an optimization process based on binary interaction dynamics inspired by kinetic models in social sciences described by spatially homogeneous Boltzmann-type equations (see \cite{partos13}).}
To this aim, let us denote by $f(v,t) \geq 0$, $v\in\RR^d$ the distribution of particles \rev{in position $v$} at time $t \geq 0$. \revn{Note that, by analogy with the classical space homogeneous Boltzmann description, we kept the  notation $v$. However, we chose to refer to this as 'position' in the search space instead of 'velocity' to employ a standard terminology in optimization algorithms.}
Without loss of generality we assume $\int_{\RR^d} f(v,t)\,dv=1$, so that $f(v,t)$ is a probability density function.
\subsection{The binary interaction process}   
For a given pair \revlp{of particles with positions} $(v,v_*)$ we consider a binary interaction process \revlp{generating the new positions $(v',v'_*)$ according to relations}
\begin{equation}
\begin{split}
v' &= v + \lambda_1(v_{\beta,\EE}(v,v_*)-v)+\lambda_2(v_{\alpha,\EE}(t)-v)+\sigma_1 D_1(v,v_*)\xi_1+\sigma_2 D_2(v)\xi_2 \\
v_*' &= v_* + \lambda_1(v_{\beta,\EE}(v_*,v)-v_*)+\lambda_2(v_{\alpha,\EE}(t)-v_*)+\sigma_1 D_1(v_*,v)\xi^*_1+\sigma_2 D_2(v_*)\xi^*_2 
\end{split}
\label{eq:binp}
\end{equation}
where $v_{\beta,\EE}(v,v_*)$, $\beta > 0$, is the {\it microscopic local estimate} of the best position
\begin{equation}\label{WbetaE}
v_{\beta,\EE}(v,v_*) = \frac{\omega_\beta^\EE(v) v + \omega_\beta^\EE(v_*)v_*}{\omega_\beta^\EE(v)+\omega_\beta^\EE(v_*)}\,, \qquad  \omega_\beta^\EE(v):=e^{-\beta\EE(v)},
\end{equation}
and $v_{\alpha,\EE}(t)$, $\alpha > 0$, is the {\it macroscopic global estimate} of the best position 
\begin{equation}\label{ValphaE}
v_{\alpha,\EE}(t)=\frac{\int_{\mathbb R^{d}}v\omega_\alpha^\EE(v)f(v,t)\,dv}{\int_{\mathbb R^{d}}\omega_\alpha^\EE(v)f(v,t)\,dv}\,, \qquad  \omega_\alpha^\EE(v):=e^{-\alpha\EE(v)}\,.
\end{equation}
 The choice of the weight function $\omega_\alpha^\EE$ in \eqref{ValphaE} comes from  the  well-known Laplace principle \cite{miller2006applied,Dembo2010,pinnau2017consensus}, a classical asymptotic method for integrals, which states that for any probability $f(v,t)$, it holds
\begin{equation}\label{Laplace}
\lim\limits_{\alpha\to\infty}\left(-\frac{1}{\alpha}\log\left(\int_{\RR^d}e^{-\alpha\EE(v)}f(v,t)\,dv \right)\right)=\inf\limits_{v\,\in\, {\rm supp}\, f(v,t)} \EE(v)\,.
\end{equation}
Similarly, in \eqref{WbetaE} as $\beta\to\infty$ the value $v_{\beta,\EE}(v,v_*)$ concentrates on the particle in the best position, namely
\begin{equation}
\lim_{\beta \to \infty} v_{\beta,\EE}(v,v_*) = \rev{ \underset{w \in \{v,v_*\}}\amin\,\EE(w)  \,,}
\label{eq:concentrate}
\end{equation}
\rev{if $\EE(v) \neq \EE(v_*)$.}
Note that, $v_{\beta,\EE}(v,v_*)$ depends on the interacting pair $(v,v_*)$, whereas $v_{\alpha,\EE}(t)$ is the same for all particles. \revlp{These quantities characterize two different dynamics where on one hand the particle pair aligns locally to $v_{\beta,\EE}(v,v_*)$ in agreement with their weighted best position and on the other hand it aligns globally to $v_{\alpha,\EE}(t)$ according to the weighted best position among all  particles.}

In \eqref{eq:binp} the scalar values $\lambda_k\geq 0$ and $\sigma_k \geq 0$, $k=1,2$ define, respectively, the strength of the relative alignment and diffusion processes, whereas the terms $\xi_k, \xi^*_k\in\RR^d$, $k=1,2$ are  vectors of i.i.d. random variables (\revlp{with arbitrary distribution}) with zero mean and unitary variance. Finally, $D_k(\cdot,\cdot)$, $k=1,2$ denote $d\times d$ dimensional diagonal matrices characterizing the stochastic exploration process. Isotropic exploration has been introduced in \cite{pinnau2017consensus} and is defined by
\begin{equation}
D_1(v,v_*)=\rev{|}v_{\beta,\EE}(v,v_*)-v\rev{|} I_d,\qquad D_2(v)=\rev{|}v_{\alpha,\EE}(t)-v\rev{|} I_d,
\label{eq:iso}
\end{equation}
with $I_d$ \rev{denoting} the $d$-dimensional identity matrix \rev{and $|\cdot|$ the euclidian norm}, whereas in the anisotropic case, introduced in \cite{carrillo2019consensus}, we have 
\begin{equation}
\begin{split}
D_1(v,v_*)&={\diag}\left\{(v_{\beta,\EE}(v,v_*)-v)_1,\ldots,(v_{\beta,\EE}(v,v_*)-v)_d\right\},\\ D_2(v)&={\diag}\left\{(v_{\alpha,\EE}(t)-v)_1,\ldots,(v_{\alpha,\EE}(t)-v)_d\right\}.
\label{eq:aniso}
\end{split}
\end{equation}  
%\revlp{
%\begin{remark}
%Throughout the manuscript we have used the notation ``$v$" to denote the position in search space of the particle. This can be confusing since in classical kinetic theory usually ``$v$" refers to the velocity of the particle and ``$x$" to the position of the particle. This choice, however, is justified by the fact that in our approach particles interact according to a spatially homogeneous Boltzmann equation where in the classical setting dynamics occurs only in velocity space.  
%\end{remark}}

%\section{Boltzmann equation and mean-field-limit}
%Let us now introduce the nonnegative distribution function $f(v,t)$ describing the number density of individuals at time $\geq 0$ with velocity $v\in \RR^d$. 
\subsection{A Boltzmann description}
\revlp{A fundamental aspect in the derivation of the corresponding evolution equation of the probability density of particles $f(v,t)$ is to determine the so-called Boltzmann collision term describing the instantaneous variations in the particles distribution. This derivation results exclusively from the binary interactions between particles given by \eqref{eq:binp} that are assumed to be uncorrelated prior to the interaction. Under this assumption, known as \emph{molecular chaos}, the collision term can be written as a multidimensional integral over the product of the distribution functions of a particle (see \cite{cer88,cer94} for further details)}. 

Thus, formally, the particle distribution satisfies a Boltzmann-type equation, \revlp{which can be conveniently written in weak form as}
\begin{equation}
\begin{split}
\frac{\partial}{\partial t} \int_{\RR^d} f(v,t)\phi(v)\,dv &= \frac12\left\langle\int_{\RR^{2d}}\left(\phi(v')+\phi(v'_*)-\phi(v)-\phi(v_*)\right)f(v,t)f(v_*,t)\,dv\,dv_*\right\rangle\\
&= \left\langle\int_{\RR^{2d}}\left(\phi(v')-\phi(v)\right)f(v,t)f(v_*,t)\,dv\,dv_*\right\rangle
\end{split}
\label{eq:Boltz}
\end{equation}
where $\phi(v)\in \rev{C}^\infty(\RR^{d})$ is a smooth function, such that
\[
\lim_{t\to 0}\int_{\RR^d} \phi(v)f(v,t)\,dv = \int_{\RR^d} \phi(v) f_0(v)\,dv
\]
with $f_0(v)$ the initial density satisfying  
\[
\int_{\RR^d} f_0(v)\,dv =1.
\]
In \eqref{eq:Boltz} we use the standard notation
\begin{equation}
\left\langle g(\xi)\right\rangle = \int_{\RR^{4d}}g(\xi)p(\xi)\,d\xi,
\end{equation}
where we used the shortcut $\xi=(\xi_1,\xi_2,\xi_1^*,\xi_2^*)$,
to denote the mathematical expectation with respect to the \rev{i.i.d.} random vectors $\xi_k, \xi^*_k$, $k=1,2$, entering the definitions of $v'$ and $v'_*$ in \eqref{eq:binp}. \revlp{As a consequence $p(\xi)=p_\xi(\xi_1)p_\xi(\xi_2)p_\xi(\xi^*_1)p_\xi(\xi^*_2)$, where $p_\xi(\cdot)$ is the common probability density function of the random vectors. }

\revlp{The Boltzmann interaction term in \eqref{eq:Boltz} quantifies the variation in the probability density, at a given time, of particles that modify their position from $v$ to $v'$ (r.h.s with negative sign) and particles that  change their value from  $v'$ to $v$ (r.h.s. with positive sign). Here, {the} expectation $\langle \cdot \rangle$ takes into account the presence of the random parameters in the microscopic interaction \eqref{eq:binp}.}

First of all, let us remark that from the binary dynamic \eqref{eq:binp} we get
\begin{equation}
\begin{split}
\langle v'+v'_*\rangle &= (1-\lambda_1-\lambda_2) (v+v_*)+2\lambda_1 v_{\beta,\EE} + 2 \lambda_2 v_{\alpha,\EE}(t),\\
\langle v'-v'_*\rangle &= (1-\lambda_1-\lambda_2) (v-v_*).
\end{split}
\end{equation}
The first equality describes the variation in the expected value of the particles positions. The second, under the assumption $\lambda_1+\lambda_2 \leq 1$,  refers to the tendency of the interaction to decrease (in mean) the distance between \revlp{positions} after the interaction. This tendency is a universal consequence of the rule \eqref{eq:binp}, in that it holds whatever distribution one assigns to $\xi$, namely to the random variable which accounts for the exploration effects.

\revlp{Before entering into a detailed analysis of the model, let us fix some notations. Throughout the paper, we will denote with $m$ and $E$ the first two moments of $f(v,t)$ 
\be 
\rev{m(t) := \int_{\RR^d} v \,f(v,t)\,dv \,,  \quad E(t) :=   \int_{\RR^d}  |v|^2\, f(v,t)\,dv\, , }
\label{eq:moments}
\ee
\rev{and the variance as}
\be
\rev{V(t) := \frac 12 \int_{\RR^d} | v- m(t)|^2\, f(v,t)\,dv = \frac12 \left(  E(t) - |m(t)|^2\right) \,.}
\label{eq:variance}
\ee
Furthermore, we will assume $\kappa$ to be a constant equal to the dimension $d$ if the isotropic exploration \eqref{eq:iso} is considered, and equal to one when the anisotropic exploration \eqref{eq:aniso} is employed.}

\section{Main properties and mean-field limit}
\label{S3}

\subsection{The case with only the microscopic best estimate}
Let us first consider the case where in the binary interaction rules \eqref{eq:binp} we assume $\lambda_2=0$ and $\sigma_2=0$. This case is particularly interesting since the dynamics is fully microscopic and therefore convergence to the global minimum will emerge from a sequel of  binary interactions which are not influenced by any macroscopic information concerning the global minimum. 

The binary interactions  can be rewritten as
\begin{equation}
\begin{split}
v' &= v + \lambda\gamma^\EE_\beta(v,v_*)(v_*-v)+\sigma D(v,v_*)\xi_1 \\
v_*' &= v_* + \lambda\gamma^\EE_\beta(v_*,v)(v-v_*)+\sigma D(v_*,v)\xi^*_1 
\end{split}
\label{eq:binp2}
\end{equation}
where, for \rev{notational} simplicity, we have set $\lambda=\lambda_1$, $\sigma=\sigma_1$, $D(v,v_*)=D_1(v,v_*)$ and
\[
\gamma^\EE_\beta(v,v_*) = \frac{\omega_\beta^\EE(v_*)}{\omega_\beta^\EE(v)+\omega_\beta^\EE(v_*)}\,.
\]
Note that, $\gamma^\EE_\beta(v,v_*)+\gamma^\EE_\beta(v_*,v)=1$, and, since $\gamma^\EE_\beta(v,v_*)\in (0,1)$, the expected support of the \revlp{positions} for $\lambda \leq 1$ is decreasing 
\[
|\langle v' \rangle| \leq (1-\lambda\gamma^\EE_\beta(v,v_*)) |v| + \lambda\gamma^\EE_\beta(v,v_*)|v_*| < \max\left\{|v|,|v_*|\right\}.
\]
Consider now, the time evolution of the \rev{expected position $m(t)$}. We have from the weak formulation \eqref{eq:Boltz} for $\phi(v)=v$
\be
\begin{split}
\rev{\frac{d m(t)}{dt}}  &= \rev{\left \la \int_{\RR^{2d}} ( v' - v) f(v,t)f(v_*,t)\,dv_*\,dv   \right \ra}
\\
&= \lambda \int_{\RR^{2d}} \gamma^\EE_\beta(v,v_*)(v_*-v)f(v,t)f(v_*,t)\,dv_* dv\\
& =  2\lambda \int_{\RR^{2d}}\gamma^\EE_\beta(v,v_*)f(v,t)f(v_*,t)v_*\,dv_*\,dv-\lambda m(t),
\end{split}
\label{eq:mom}
\ee
\rev{where we made use of the fact that $\gamma_\beta^\EE(v, v_*) + \gamma_\beta^\EE(v_*,v) = 1$, from which follows}
\[\rev{
m(t) = \int_{\RR^{2d}} \gamma^\EE_\beta(v,v_*)(v_*-v)f(v,t)f(v_*,t)\,dv_* dv + \int_{\RR^{2d}} \gamma^\EE_\beta(v_*,v)(v_*-v)f(v,t)f(v_*,t)\,dv_* dv\,.
}\]
It is easy to verify that the above equation admits as steady state any Dirac delta distribution of the form $f^\infty(v)=\delta(v-\bar{v})$, since $\gamma^\EE_\beta(\bar{v},\bar{v})=1/2$, $\forall\,\, \bar{v}\in \RR^d$.  In general, any symmetric function $\gamma^\EE_\beta(v,v_*) \rev{=\gamma^\EE_\beta(v_*,v)}$ would preserve the average \revlp{position}, and it is therefore the asymmetric behavior of this function based on the choice of the best value in the binary interaction that will asymptotically lead  to the global minimum in the system. Note that, equation \eqref{eq:mom} is not closed. 

%Before considering the behavior of higher order moments it is useful to state the following identity.
%\begin{lemma}
%For any symmetric function $\zeta(v,v_*)=\zeta(v_*,v)$ we have
%\begin{equation}
%\int_{\RR^d} \gamma^\EE_\beta(v,v_*) \zeta(v,v_*) f(v)f(v_*)\, dv dv_* = \frac12\int_{\RR^d} \zeta(v,v_*) f(v)f(v_*)\, dv dv_*.
%\label{eq:id}
%\end{equation}
%\label{le:1}
%\end{lemma}
%\begin{proof}
%Switching the notation between $v$ and $v_*$ in the integration we have 
%\[
%\int_{\RR^d} \gamma^\EE_\beta(v,v_*) \zeta(v,v_*) f(v)f(v_*)\, dv dv_* = \int_{\RR^d} \gamma^\EE_\beta(v_*,v) \zeta(v_*,v) f(v_*)f(v)\, dv_* dv.
%\]
%Now, using the fact that $\gamma^\EE_\beta(v_*,v) = 1 -\gamma^\EE_\beta(v,v_*)$ we can write 
%\[
%\int_{\RR^d} \gamma^\EE_\beta(v_*,v) \zeta(v_*,v) f(v_*)f(v)\, dv_* dv=\int_{\RR^d} (1-\gamma^\EE_\beta(v,v_*)) \zeta(v_*,v) f(v_*)f(v)\, dv_* dv,
%\]
%which, switching again the integration notation in the last integral and using the symmetry of $\zeta(\cdot,\cdot)$, yields the desired result \eqref{eq:id}.
%\end{proof}
\rev{
In order to analyze the large time behavior of $f(v,t)$, we introduce the following boundedness assumption on $\EE(v)$.}

\begin{assumption}
\label{a:bound}
Let us assume $\EE(w)$ positive and for all $w \in \RR^d$
\[ \underline \EE := \inf_{v\in\RR^d}\EE(v) \leq \EE(w) \rev{\leq} \sup_{v \in \RR^d}\EE(v)=: \overline \EE \,. \]
\end{assumption}
Under this assumption, it is possible to show that, when the alignment and exploration strengths satisfy suitable conditions, the particle system concentrates as it evolves.

\begin{proposition} \label{p:locvar}
\rev{Let $f(v,t)$ be a weak solution of equation \eqref{eq:Boltz} with initial data $f_0$ and binary interaction described by the system \eqref{eq:binp2}. If $\EE$ satisfies Assumption \ref{a:bound} and
 $\beta$ is sufficiently large, it holds}
\rev{ 
\be
\frac{d V(t)}{dt} \leq - \left( \frac{\lambda}{C_{\beta,\EE}} - \lambda^2  - \sigma^2 \kappa\right)  V(t) \,,
\label{eq:dVbeta}
\ee
for all $t>0$, where $C_{\b,\EE} := e^{\b(\overline \EE -\underline \EE)}$.}

\end{proposition}
We start the proof by presenting an auxiliary result.

\begin{lemma}\label{lemmaxi} \rev{If $\beta$ is sufficiently large,} it holds
\be 
\left( \gamma_\b^\EE (v,v_*) \right)^2 \leq \rev{\left (1- \frac1{C_{\beta, \EE}} \right)} \gamma_{2\b}^\EE\,\rev{(v,v_*)} \,,
\ee
\rev{where $C_{\b,\EE} := e^{\b(\overline \EE -\underline \EE)}$.}
\end{lemma}

\begin{proof}

\rev{We start by rewriting $( \gamma_\b^\EE (v,v_*))^2 $ as}
\[
\begin{split}
\left( \gamma_\b^\EE (v,v_*) \right)^2 &=  \frac{e^{-2\b \EE(v_*)}}{\left(e^{-\beta \EE(v)} + e^{-\beta\EE(v_*)}\right)^2} =  
\frac{e^{-2\b \EE(v_*)}}{e^{-2\b \EE(v)} + e^{-2\beta\EE(v_*)}} 
\frac{e^{-2\b \EE(v)} + e^{-2\b\EE(v_*)}}{\left(e^{-\b \EE(v)} + e^{-\b\EE(v_*)}\right)^2}\\
& = \gamma_{2\b}^\EE\rev{(v,v_*)} \frac{e^{-2\b \EE(v)} + e^{-2\b\EE(v_*)}}{\left(e^{-\b \EE(v)} + e^{-\b\EE(v_*)}\right)^2} \rev{=:  \gamma_{2\b}^\EE(v,v_*)\zeta_\b^\EE(v,v_*)  }
\end{split}
\]
\rev{and further rewrite $\zeta_\beta^\EE(v,v_*) $ as}
\be \begin{split} \notag
\rev{\zeta_\b^\EE(v,v_*) } &= \frac{e^{-2\b \EE(v)} + e^{-2\b\EE(v_*)}}{\left(e^{-\b \EE(v)} + e^{-\b\EE(v_*)}\right)^2} = \frac{e^{-2\b \EE(v_*)}\left(1+ e^{-2\b( \EE(v) - \EE(v_*))} \right)}{e^{-2\b \EE(v_*)}\left(1+ e^{-\b( \EE(v) - \EE(v_*)) } \right)^2} \\
&\rev{= \frac{1+ e^{-2\b( \EE(v) - \EE(v_*))}}{\left(1+ e^{-\b( \EE(v) - \EE(v_*)) } \right)^2}\,.}
\end{split} \ee

\rev{One can verify that $\zeta_\b^\EE(v,v_*)$ attains its maximum value when the difference $|\EE(v) - \EE(v_*)|$ is maximized, from which follows} 
\rev{
\be\notag
\zeta_\b^\EE(v,v_*)\, \rev{\leq}\, \frac{1 + e^{-2\b (\overline \EE - \underline \EE)}} {\left(1+e^{-\b (\overline \EE - \underline \EE)} \right)^2} \rev{\,= \frac{1 + C^2}{(1+C)^2}\,,}
\ee}
\rev{where we denoted for simplicity $(C_{\beta,\EE})^{-1}=:C$. We note that $C \to 0$ as $\beta \to \infty$. Finally, as $\beta \to \infty$} 
\be
\rev{
\begin{split}\notag
1- \zeta_\beta^\EE(v,v_*) -  (C_\beta^\EE)^{-1} & = 
 1- \frac{1 + C^2}{(1+C)^2} - C = \frac{C + o(C) }{(1+C)^2} \geq 0,
\end{split}
}
\ee
\rev{if $\beta$ is sufficiently large. This proves the assertion.}

\end{proof}

\begin{proof}[Proof of Proposition \ref{p:locvar}]
\rev{From the definition of $E(t)$, and the weak formulation \eqref{eq:Boltz},} we can compute
\be
\begin{split}
\frac{d E(t)}{dt}
& =\left\langle\int_{\RR^{2d}}\left({\rev{|}v'\rev{|}}^2-\rev{|}v\rev{|}^2\right)f(v,t)f(v_*,t)\,dv\,dv_*\right\rangle\\
& = \lambda^2 \int_{\RR^{2d}} \gamma^\EE_\beta(v,v_*)^2\rev{|}v_*-v\rev{|}^2 f(v,t)f(v_*,t) \, dv\, dv_*\\
&+2\lambda \int_{\RR^{2d}} \gamma^\EE_\beta(v,v_*) v\rev{\cdot} (v_*-v)
 f(v,t)f(v_*,t) \, dv\, dv_*\\
&+  \sigma^2 \sum_{i=1}^d \int_{\RR^{2d}}  D_{ii}(v,v_*)^2 f(v,t)f(v_*,t) \, dv\, dv_*\,,
\label{eq:energy1}
\end{split}
\ee
\rev{where with $D_{ii}$ we denote the $i$th- diagonal element of the matrix $D$.}
From 
\[ \frac{d}{dt}V(t) = \frac{1}{2} \frac{d}{dt} \rev{E}(t) - m(t) \frac{d}{dt} m(t)\]
and the moment derivative \eqref{eq:mom}, we recover
\be
\begin{split}
\frac{d V(t)}{dt}
& =\left\langle\int_{\RR^{2d}}\left(\rev{|}v'\rev{|}^2-\rev{|}v\rev{|}^2\right)f(v,t)f(v_*,t)\,dv\,dv_*\right\rangle \rev{\,-\, m(t) \frac{d}{dt}m(t)}   \\
& = \frac{\lambda^2}{2} \int_{\RR^{2d}} \gamma^\EE_\beta(v,v_*)^2 \rev{|} v_*-v \rev{|}^2 f(v,t)f(v_*,t) \, dv\, dv_*\\
&+\lambda \int_{\RR^{2d}} \gamma^\EE_\beta(v,v_*) (v- m(t)) \rev{\cdot} (v_*-v)
 f(v,t)f(v_*,t) \, dv\, dv_*\\
&+\frac{\sigma^2}{2} \sum_{i=1}^d \int_{\RR^{2d}}  D_{ii}(v,v_*)^2 f(v,t)f(v_*,t) \, dv\, dv_* =: I_1 + I_2 + I_3
\end{split}
\ee 

\rev{Thanks to the relation, $\gamma_\beta^\EE(v,v_*) + \gamma_\beta^\EE(v_*,v) = 1$, we note that for any symmetric function $\psi(v, v_*)=\psi(v_*,v)$ it holds}
\be
\int_{\RR^{2d}} \gamma^\EE_\beta(v,v_*) \psi(v,v_*) f(v,t)f(v_*,t) \, dv\, dv_* = \frac 12 \int_{\RR^{2d}} \psi(v,v_*) f(v,t)f(v_*,t) \, dv\, dv_*\,.
\label{eq:symm}
\ee
\rev{It follows} that $I_1$ and $I_3$ can be bounded as 
\begin{align}
I_1 &\leq \frac{\lambda^2}{2} \int_{\RR^{2d}} \gamma^\EE_\beta(v,v_*) \rev{|}v_*-v\rev{|}^2 f(v,t)f(v_*,t) \, dv\, dv_*  =  \lambda^2  V(t) \label{eq:I1}\\
I_3 & \leq \frac{\sigma^2}{2} \kappa \int_{\RR^{2d}}\rev{ \gamma^\EE_\beta(v,v_*)} |v_* - v|^2 f(v,t)f(v_*,t) \, dv\, dv_*  \rev{=}  \sigma^2\kappa V(t)\,,
\label{eq:I2}
\end{align}
where \rev{we recall that} $\kappa=d$ in the isotropic case \eqref{eq:iso}, and $\kappa=1$ in the anisotropic case \eqref{eq:aniso}.  We compute by means on Young's inequality
\be
\begin{split}
I_2&=\lambda \int_{\RR^{2d}} \gamma^\EE_\beta(v,v_*) (v- m(t)) \rev{\cdot} (v_*-v)
 f(v,t)f(v_*,t) \, dv\, dv_*\\
&\leq- \lambda\int_{\RR^{2d}} \gamma_\b^\EE (v,v_*) |v-v_*|^2 f(v,t) f(v_*,t) dv dv_* + \frac{\lambda}{2} \int_{\RR^{2d}} |v_*-m(t)|^2 f(v,t) f(v_*,t) \,dv\, dv_* \\
&\quad+\frac{\lambda}{2} \int_{\RR^{2d}} \left( \gamma_{\b}^\EE (v,v_*)\right)^2 |v-v_*|^2 f(v,t) f(v_*,t) \,dv \,dv_*\,. %\\
%&\leq -2\lambda V(t) + \lambda V(t) + \frac{\lambda}{2} \int_{\RR^{2d}} \left( \gamma_\b^\EE (v,v_*)  \right)^2 |v-v_*|^2 f(v,t) f(v_*,t) dv dv_* \, .
\end{split}
\ee
\rev{By applying Lemma \ref{lemmaxi} one can bound the last term as
\[\int_{\RR^{2d}} \left( \gamma_{\b}^\EE (v,v_*)\right)^2 |v-v_*|^2 f(v,t) f(v_*,t) dv dv_* \leq \rev{\left(1- \frac1{C_{\beta, \EE}}\right)}\int_{\RR^{2d}} \rev{ \gamma_{2\b}^\EE (v,v_*)} |v-v_*|^2 f(v,t) f(v_*,t) dv dv_*\,.\] 
Finally, we use again relation \eqref{eq:symm} to obtain} 

\[
I_2 \leq -2 \lambda V(t) +\lambda V(t) +\lambda \rev{\left(1- \frac1{C_{\beta, \EE}}\right)} V(t) = -\frac{\lambda}{ C_{\b,\EE}} V(t)
\]
and hence, together with \eqref{eq:I1} and \eqref{eq:I2}, \rev{we get \eqref{eq:dVbeta}}.
\end{proof}

\rev{
\begin{corollary}
\label{c:1}
Under the assumptions of Proposition \ref{p:locvar}, if $\lambda$ and $\sigma$ satisfy the condition
\be
\frac{\lambda}{C_{\beta,\EE}} - \lambda^2  - \sigma^2 \kappa >0
\label{eq:cond}
\ee 
then there exits $\tilde v \in \RR^d$ such that 
$m(t) \rightarrow \tilde v$, $V(t) \to 0$ as $t \to \infty.$
\end{corollary}
\begin{proof}
By applying Grönwall's inequality to equation \eqref{eq:dVbeta}, we obtain the decay estimate
\be
V(t) \leq V(0) e^{-\mu t} \quad \text{with}\quad  \mu:=\frac{\lambda}{C_{\beta,\EE}} - \lambda^2  - \sigma^2 \kappa >0\,,
\ee
which implies $V(t) \to 0$ as $t \to \infty$.
From the weak formulation \eqref{eq:Boltz}, 
\[
\begin{split}
\left | \frac{d m(t)}{dt}\right| &= \left | \lambda \int_{\RR^2} \gamma_\b^\EE(v,v_*) (v_* - v) f(v,t)f(v_*,t)\,dv\,dv_*  \right |\\
& \leq \lambda \int_{\RR^2} | v_* - v| f(v,t)f(v_*,t)\,dv\,dv_* \\
& \leq \lambda \left( \int_{\RR^2} | v_* - v|^2 f(v,t)f(v_*,t)\,dv\,dv_*\right)^{\frac12} \leq 2\lambda \sqrt{V(t)} \leq 2\lambda \sqrt{V(0)}e^{-\frac{1}{2}\mu t}\,,
\end{split}
\]
where we used Jensen's inequality to have an estimate in terms of the variance. The above proves that $dm(t)/dt \in L^1(0, \infty)$ and, hence, that there exists a point $\tilde v\in \RR^d$ such that
\be
\tilde v = m(0) + \int_0^{\infty} \frac{d m(t)}{dt} \, dt = \lim_{t\to\infty} m(t)\,.
\ee
\end{proof}
}

%We remark that condition \eqref{eq:cond} on $\sigma$ and $\lambda$ is sufficient to guarantee that $V(t)$ decays to zero as $t\to \infty$, and thus the solution concentrates over \revlp{the asymptotic value of the expected position. 
\rev{
\begin{remark}
Clearly, the asymptotic value $\tilde v$ in general is not known. We will discuss in Section \ref{S4} appropriate conditions under which $\EE(\tilde v)$ can be considered a good approximation of $\inf_{v\in \RR^d} \EE(v)$. It should be noted that, condition \eqref{eq:cond} becomes rather restrictive for large values of $\beta$. However, in the mean-field scaling such a condition becomes less stringent as observed in Remark \ref{R33}.
%On the other hand, as we will see in Section \ref{S4}, in the case with only macroscopic best estimate the solution concentrates under much less restrictive conditions (see Theorem \ref{t:2}).  
Additionally, when both processes for localizing the minimum, microscopic best and macroscopic best, are activated simultaneously the convergence conditions are much less stringent and correspond to those of the macroscopic best dynamics as shown at the end of Section \ref{S4} (see Theorem \ref{t:3}). From a physical point of view, this reflects the tendency of the binary dynamics based on the microscopic best to favor exploration over concentration when compared to the corresponding binary dynamics based on macroscopic best.
\end{remark}}

\subsection{\rev{The case with only the macroscopic best estimate}}

The case where \rev{the macroscopic best estimate contributes alone} to the particle search dynamics can be analyzed following the same methodology of the previous section. 

The binary interactions \rev{now read}
\begin{equation}
\begin{split}
 \quad v'    &= v + \lambda(\rev{v_{\alpha,\EE}(t)}-v)+\sigma \rev{D(v)\xi_2} \\
 \quad v_*' &= v_* + \lambda(\rev{v_{\alpha,\EE}(t)}-v_*)+\sigma \rev{D(v_*)\xi^*_2}
\end{split}
\label{eq:binp3}
\end{equation}
where we have set $\lambda=\lambda_2$, $\sigma=\sigma_2$, $\rev{D(v)=D_2(v)}$ \rev{and $\lambda_1 = \sigma_1 = 0$}.

Again, \revlp{the expected position} is not conserved by the dynamics
\be
\rev{\frac{d m(t)}{dt}} = \lambda (\rev{v_{\alpha,\EE}(t)}-m(t)), 
\label{eq:mom2}
\ee
and describes a relaxation towards the estimated global minimum $\rev{v_{\alpha,\EE}(t)}$. 

%Also, if $\EE$ satisfies Assumption \ref{a:bound}, the variance decays exponentially. As before, we set $C_{\a,\EE} := e^{\a(\overline \EE - \underline \EE)}$.
\rev{As in the case with only microscopic interaction, we can derive an upper bound for the variance derivative.}

\begin{proposition}
\label{p:macrovar}
\rev{Let $\EE$ satisfy Assumption \ref{a:bound} and $f(v,t)$ be a weak solution of the Boltzmann equation \eqref{eq:Boltz} were the binary interaction is described by \eqref{eq:binp3} .}
For all $\a>0$ and $t>0$,  
\rev{
\be 
\frac{d V(t)}{dt} \leq- \left( 2 \lambda - 2\frac{e^{-\alpha\underline \EE}}{\|\omega_\a^\EE\|_{L^1(f(\cdot,t))}}  (\lambda^2 + \kappa\sigma^2) \right) V(t)\,.
\label{eq:dValpha}
\ee
}
\end{proposition} 
\begin{proof}

\rev{We start by noting that, according to \eqref{eq:binp3}},
\be
\rev{
\begin{split}
 \left\la |v'|^2 \right\ra &= \left\la | v + \lambda(\rev{v_{\alpha,\EE}}(t)-v)+\sigma D(v)\xi_2 |^2\right\ra \\
& =|v|^2 + \lambda^2 |\rev{v_{\alpha,\EE}}(t)-v|^2+2 \lambda\, v \cdot (\rev{v_{\alpha,\EE}}(t)-v) + \sigma^2 \sum_{i=1}^d D_{ii}(v)^2 \,,
\label{eq:deltaE_2}
\end{split}
}
\ee
\rev{where we used that the $\la \xi_2 \ra =0$ and $\la |\xi_2|^2 \ra= 1$. As before, we compute} %add that we can interchange expectation and \int
\rev{
\begin{eqnarray}
\nonumber
\frac{dE(t)}{dt} &=& \left \la \int_{\RR^{2d}} (| v'|^2 - |v|^2) f(v,t) f(v_*,t)\,dv\, dv_* \right \ra 
\\[-.2cm]
\label{eq:E2}
\\[-.3cm]
\nonumber
 & =& \int_{\RR^{2d}}\left( \lambda^2|\rev{v_{\alpha,\EE}}(t)-v|^2 
+2\lambda  v \cdot (\rev{v_{\alpha,\EE}}(t)-v) + \sigma^2 \sum_{i=1}^d  D_{ii}(v)^2 \right ) f(v,t)f(v_*,t) \, dv\, dv_*\,,
\end{eqnarray}
}
\rev{and the variance time evolution}
\be
\begin{split}
\frac{d V(t)}{dt}
& = \frac{\lambda^2}{2}\int_{\RR^{2d}} \rev{|v_{\alpha,\EE}}(t)-v\rev{|}^2  f(v,t)f(v_*,t) \, dv\, dv_*\\
&+\lambda \int_{\RR^{2d}} (v-m(t))\rev{\cdot}(\rev{v_{\alpha,\EE}}(t)-v)  f(v,t)f(v_*,t) \, dv\, dv_*\\
&+ \frac{\sigma^2}{2} \sum_{i=1}^d \int_{\RR^{2d}}  D_{ii}\rev{(v)}^2 f(v,t)f(v_*,t) \, dv\, dv_*\,.
\label{eq:energy2}
\end{split}
\ee
Thanks to the identity
\begin{equation*}
\begin{split}
\int_{\RR^{d}}(v-m(t)) \rev{\cdot} (\rev{v_{\alpha,\EE}}(t)-v) f(v,t)\,dv & = 
\rev{ \int_{\RR^d} ( v \cdot \rev{v_{\alpha,\EE}}(t) - m(t) \cdot v_{\alpha,\EE}(t) - |v|^2 + v \cdot m(t) ) f(v,t)\,dv } \\
&= \int_{\RR^{d}} ( - |v|^2 + \rev{|m(t)|^2}) f(v,t)\,dv,
\end{split}
\end{equation*}
we note that the second term \rev{of \eqref{eq:energy2}} is equal to $-2\lambda V(t)$. 

\rev{We recall that, from \eqref{ValphaE}, $v_{\alpha,\EE}(t)$ is defined as
\[ v_{\alpha,\EE}(t) = \frac{\int_{\mathbb R^{d}}v\omega_\alpha^\EE(v)f(v,t)\,dv}{\int_{\mathbb R^{d}}\omega_\alpha^\EE(v)f(v,t)\,dv}  =  \int_{\mathbb R^{d}}v\, \frac{e^{-\alpha\EE(v)}}{\|\omega_\a^\EE\|_{L^1(f(\cdot,t))}}f(v,t)\,dv\]}

The remaining terms in \eqref{eq:energy2} can then be estimated by pointing out that
\be
\begin{split}
\int_{\RR^d}|v_{\alpha,\EE}\rev{(t)} - v|^2 f(v,t) dv &\rev{\leq \int_{\RR^{2d}} |v-w|^2 \frac{e^{-\a\EE(w)}}{\rev{\|\omega_\a^\EE\|_{L^1(f(\cdot,t))}}}f(v,t) f(w,t)dvdw}\\
&\leq \rev{2\frac{e^{-\a\underline\EE}}{\|\omega_\a^\EE\|_{L^1(f(\cdot,t))}}}\int_{\RR^d} |v-m(t)|^2 f(v,t) dv
\label{eq:valpha-v}
\end{split}
\ee
thanks to Jensen's inequality. % We refer to \ref{CBO sphere II} for more details on the above inequality.
Lastly, we obtain \rev{the desired upper bound}
\be
\begin{split}
\frac{dV(t)}{dt} &\leq \rev{\lambda^2  \frac{e^{-\a\underline\EE}}{\|\omega_\a^\EE\|_{L^1(f(\cdot,t))}}} \int_{\RR^d} |v-m(t)|^2 f(v,t) dv - \lambda \int_{\RR^d} |v-m(t)|^2 f(v,t) dv\\
&+ \rev{\sigma^2} \kappa  \rev{\frac{e^{-\a\underline\EE}}{\|\omega_\a^\EE\|_{L^1(f(\cdot,t))}}} \int_{\RR^d} |v-m(t)|^{\rev{2}}f(v,t) dv \\
&\leq -\left( 2\lambda-  \rev{2\frac{e^{-\a\underline\EE}}{\|\omega_\a^\EE\|_{L^1(f(\cdot,t))}} (\lambda^2 + \kappa \sigma^2)}\right) V(t)\,.
\end{split}
\label{eq:Vest2}
\ee
\end{proof}

\rev{
\begin{remark}
We note that, by applying $\|\omega_\a^\EE\|_{L^1(f(\cdot,t))} \geq e^{-\a \overline \EE}$ to \eqref{eq:dValpha} one gets an analogous condition, as in Corollary \ref{c:1} with $C_{\b, \EE}$ replaced by $C_{\a, \EE}:= e^{\a(\overline \EE - \underline \EE)}$, under which the solution $f$ concentrates around a point $\tilde v \in \RR^d$.
%\be
%\frac{d V(t)}{dt} \leq -\left( 2 \lambda - 2C_{\a,\EE}  (\lambda^2 + \kappa \sigma^2) \right) V(t)
%\ee
%where $C_{\a, \EE}:= e^{\a(\overline \EE - \underline \EE)}$. This allows to derive a condition for $\lambda$ and $%\sigma$, namely
%\be
%\label{eq:cond2}
%\frac{\lambda}{C_{\a,\EE}} - \lambda^2 - \sigma^2 \kappa >0\,,
%\ee
% under which the solution $f$ concentrates around a point $\tilde v \in \RR^d$, in the sense of Corollary \ref{c:1}. 
However, as we will see in Section \ref{S4}, taking into account the time evolution of $\|\omega_\a^\EE\|_{L^1(f(\cdot,t))}$ a weaker condition can be obtained, which avoids the limitations induced by large values of $\a$.
\end{remark}
}

%% add comments
%\revlp{Finally, let us remark that conditions \eqref{eq:cond}, \eqref{eq:cond2} and \eqref{eq:cond3} become rather restrictive for large values of $\beta$ and $\alpha$. We will see in the next section how these conditions relax when a suitable scaling is considered. From a physical point of view, this reflects the tendency of the Boltzmann dynamics to favor exploration over concentration when compared to the corresponding mean-field dynamics. Note also, as expected, that in the case of isotropic noise we  have a dimensional effect which is not present in the anisotropic case \cite{carrillo2018analytical,carrillo2019consensus}.}  

\subsection{The mean-field scaling limit}
Let us consider, for \rev{the sake of notational simplicity}, the case with only the microscopic binary estimate. We introduce the following scaling 
\be
t \to \frac{t}{\varepsilon},\qquad \lambda \to \lambda\varepsilon,\qquad \sigma \to \sigma\sqrt{\varepsilon}.
\label{eq:scaling}
\ee
The scaling \eqref{eq:scaling}, allows to recover in the limit the contributions due both to alignment and random exploration by diffusion. Other scaling limits can be considered, which are diffusion dominated or alignment dominated. As we shall see, derivation of mean-field CBO models is possible only under this choice of scaling.

\revlp{To illustrate this, let us consider the decay of the variance which is given by \eqref{eq:dVbeta}. If we now rescale time as $t \to {t}/{\varepsilon}$ we get 
\be
\frac{dV(t)}{dt} \leq - \frac1{\varepsilon}\left(\frac{\lambda}{C_{\beta,\EE}}-{\lambda^2}-{\sigma^2}\kappa\right) V(t).
\label{eq:sv}
\ee
Letting now $\varepsilon\to 0$ in order to preserve the behavior of the variance and both alignment and diffusion dynamics we need to assume both $\lambda$ and $\sigma^2$ as $O(\varepsilon)$. This argument shows that the choice of the scaling \eqref{eq:scaling} is of paramount importance to get mean-field asymptotics which maintain memory of the microscopic interactions and concentration effects.
}

\revlp{In the remainder of this section, we shall present the formal derivation of the mean-field limit, starting from weak form of the Boltzmann equation \eqref{eq:Boltz} under the scaling \eqref{eq:scaling} which leads to the microscopic binary interactions 
\begin{equation}
\begin{split}
v' &= v + \varepsilon\lambda\gamma^\EE_\beta(v,v_*)(v_*-v)+\sqrt{\varepsilon}\sigma D(v,v_*)\xi_1 \\
v_*' &= v_* + \varepsilon\lambda\gamma^\EE_\beta(v_*,v)(v-v_*)+\sqrt{\varepsilon}\sigma D(v_*,v)\xi^*_1. 
\end{split}
\label{eq:binp2s}
\end{equation}
}
For small values of $\varepsilon>0$ we have $v' \approx v$ and  we can consider the multidimensional Taylor expansion
\[
\phi(v')=\phi(v)+(v'-v)\cdot \nabla_v \phi(v) + \sum_{|\eta|=2} (v'-v)^\eta \frac{\partial^\eta\phi(v)}{\eta!}+\sum_{|\eta|=3} (v'-v)^\eta \frac{\partial^\eta\phi(\rev{\hat v})}{\eta!},
\]
where we used the multi-index notation $|\eta|=\eta_1+\ldots+\eta_d$, $\eta!=\eta_1!\ldots \eta_d!$,
\[
\partial^\eta \phi(v)= \frac{\partial^{|\eta|}}{\partial^{\eta_1} v_1\ldots \partial^{\eta_d} v_d}\rev{\phi(v)},
\quad (v'-v)^\eta = (v_1'-v_1)^{\eta_1}\cdots (v'_d-v_d)^{\eta_d},
\]
and $\rev{\hat v} = \theta v + (1-\theta) v'$, for some $\theta \in (0,1)$. We refer to \cite{partos13} for an extensive discussion on this kind of asymptotic limits leading from a Boltzmann dynamic to the corresponding mean-field behavior. Here, we limit ourselves, to observe that form an algorithmic viewpoint this corresponds to increase the frequency of binary interactions by reducing the strength of each single interaction. 

Now \eqref{eq:Boltz}, under the scaling \eqref{eq:scaling}, can be written as
\begin{equation}
\begin{split}
\frac{\partial}{\partial t} \int_{\RR^d} f(v,t)\phi(v)\,dv &= \frac1{\varepsilon}\left\langle\int_{\RR^{2d}}\left(\phi(v')-\phi(v)\right)f(v,t)f(v_*,t)\,dv\,dv_*\right\rangle\\
&= \lambda \int_{\RR^{2d}}\gamma^\EE_\beta(v,v_*) \nabla_v \phi(v)\cdot (v_*-v)f(v,t)f(v_*,t)\,dv\,dv_*\\
& + \varepsilon \frac{\lambda^2}{2} \int_{\RR^{2d}}(\gamma^\EE_\beta(v,v_*))^2\sum_{|\eta|=2} (v_*-v)^\eta \frac{\partial^\eta\phi(v)}{\eta!}f(v,t)f(v_*,t)\,dv\,dv_* \\
& + \frac{\sigma^2}{2} \int_{\RR^{2d}}\sum_{i=1}^d D_{ii}^2(v,v_*)\frac{\partial^2 \phi(v)}{\partial v_i^2}f(v,t)f(v_*,t)\,dv\,dv_* \\
& + O(\sqrt{\varepsilon})
\end{split}
\label{eq:Boltz2}
\end{equation}
%Since $\gamma^\EE_\beta(v,v_*) \in (0,1)$ we have 
%\[
%|I(v,t)| \leq \|\phi\|_{C^2_0} T_f(0),
%\]
%and, 
Under suitable boundedness assumptions on moments up to order three, we can formally pass to the limit $\varepsilon \to 0$ to get the weak form
\begin{equation}
\begin{split}
\frac{\partial}{\partial t} \int_{\RR^d} f(v,t)\phi(v)\,dv 
=& \, \lambda \int_{\RR^{2d}}\gamma^\EE_\beta(v,v_*) \nabla_v \phi(v)\cdot (v_*-v)f(v,t)f(v_*,t)\,dv\,dv_*\\
& + \frac{\sigma^2}{2} \int_{\RR^{2d}}\sum_{i=1}^d D_{ii}^2(v,v_*)\frac{\partial^2 \phi(v)}{\partial v_i^2}f(v,t)f(v_*,t)\,dv\,dv_*. 
\end{split}
\label{eq:Boltz3}
\end{equation}
This implies that $f$ satisfies the mean-field limit equation
\be
\begin{split}
\frac{\partial f(v,t)}{\partial t} + \lambda \nabla_v \cdot &\left(f(v,t)\int_{\RR^{d}}\gamma^\EE_\beta(v,v_*)(v_*-v)f(v_*,t)\,dv_* \right)\\ 
&= \frac{\sigma^2}{2}\sum_{i=1}^d \frac{\partial^2}{\partial v_i^2} \left(f(v,t)\int_{\RR^{d}} D_{ii}^2(v,v_*) f(v_*,t)\,dv_*\right).
\end{split}
\label{eq:mf}
\ee
The explicit expressions of the diffusion terms are given below for the isotropic case
\be
\int_{\RR^{d}} D_{ii}^2(v,v_*) f(v_*,t)\,dv_* = \sum_{j=1}^d \int_{\RR^{d}} \gamma^\EE_\beta(v,v_*)^2 (v_{*,j}-v_j)^2 f(v_*,t)\,dv_*
\ee
and the anisotropic one
\be
\int_{\RR^{d}} D_{ii}^2(v,v_*) f(v_*,t)\,dv_* = \int_{\RR^{d}} \gamma^\EE_\beta(v,v_*)^2 (v_{*,i}-v_i)^2 f(v_*,t)\,dv_*.
\ee 

%Thanks to the scaling \eqref{eq:scaling}, concentration condition \eqref{eq:cond} reduces to
%\be
%\sigma^2 \leq \frac{\lambda(1-\delta)}{\kappa}.
%\ee

In the general case, by analogous computations, under boundedness assumptions on moments, in the limit $\varepsilon\to 0$ we get
the weak form
\begin{equation}
\begin{split}
\frac{\partial}{\partial t} \int_{\RR^d} f(v,t)\phi(v)\,dv 
=\, & \lambda_1 \int_{\RR^{2d}}\gamma^\EE_\beta(v,v_*) \nabla_v \phi(v)\cdot (v_*-v)f(v,t)f(v_*,t)\,dv\,dv_*\\
& + \lambda_2 \int_{\RR^{d}}\nabla_v \phi(v)\cdot (\rev{v_{\alpha,\EE}(t)}-v)f(v,t)\,dv\\
& + \frac{\sigma_1^2}{2} \int_{\RR^{2d}}\sum_{i=1}^d D_{1,ii}^2(v,v_*)\frac{\partial^2 \phi(v)}{\partial v_i^2}f(v,t)f(v_*,t)\,dv\,dv_*\\
& + \frac{\sigma_2^2}{2} \int_{\RR^{2d}}\sum_{i=1}^d D_{2,ii}^2(v)\frac{\partial^2 \phi(v)}{\partial v_i^2} 
f(v,t)\,dv,
\end{split}
\label{eq:Boltz3b}
\end{equation}
which corresponds to the mean-field limit equation
\begin{eqnarray}
\nonumber
\frac{\partial f(v,t)}{\partial t} &+& \lambda_1 \nabla_v \cdot \left(f(v,t)\int_{\RR^{d}}\gamma^\EE_\beta(v,v_*)(v_*-v)f(v_*,t)\,dv_* \right) 
 + \lambda_2 \nabla_v \cdot \left(f(v,t)(\rev{v_{\alpha,\EE}(t)}-v) \right)\\[-.2cm]
\label{eq:mfb}
 \\[-.3cm]
 \nonumber
&=& \frac{\sigma_1^2}{2}\sum_{i=1}^d \frac{\partial^2}{\partial v_i^2} \left(f(v,t)\int_{\RR^{d}} D_{1,ii}^2(v,v_*) f(v_*,t)\,dv_*\right)
+\frac{\sigma_2^2}{2}\sum_{i=1}^d \frac{\partial^2}{\partial v_i^2} \left(f(v,t)D_{2,ii}^2(v)\right).
\end{eqnarray}
\rev{
\begin{remark}
\label{R33}
System \eqref{eq:mfb} generalizes the notion of CBO model to the case where a local interaction is taken into account. Additionally, let us remark that from the scaling \eqref{eq:scaling} in the mean field limit we have the analogous of Proposition \ref{p:locvar} and \ref{p:macrovar} where now the $\lambda^2$ terms disappear, making the corresponding concentration conditions less restrictive.    
\end{remark}
}

%\revlp{Finally, let us mention here some recent results concerning the rigorous mean-field limit of the CBO system, namely the mean-field equation \eqref{eq:mfb} in absence of microscopic interactions \cite{}.}  

%%%%%%%%%%%%%%%%% CONVERGENCE RESULTS %%%%%%%%%%%%%%%%%%%%%%%%%%%%%%%%%%%

\section{Convergence to the global minimum}% Error estimate?
\label{S4}

In this section, we will attempt to understand under which conditions we can assume $\displaystyle\lim_{t\rightarrow\infty} \EE\left(m(t)\right)$ to be a good approximation of $\underline \EE:= \min_{v\in\RR^d}\EE(v)$.

In order to do so, we will investigate the large-time behavior of the solution $f(v,t)$ to the \rev{Boltzmann} equation \rev{\eqref{eq:Boltz}}. Here, we will \rev{first} limit ourselves to the case where only the microscopic best estimate occurs during the interactions \rev{and then study the case where only the macroscopic best estimate occurs.} %For the case where the particles are driven towards the best global estimate, we refer to the convergence results regarding the CBO model \cite{carrillo2018analytical,carrillo2019consensus,fhps20-2}, upon which our analysis is built on.

\subsection{\rev{The case with only the microscopic best estimate}}

\rev{In order to study} the fully microscopic dynamics, let us set $\lambda_2 = \sigma_2 =0$ and $\lambda = \lambda_1$, $\sigma=\sigma_1$. Throughout this section we assume $\EE$ to satisfy Assumption \ref{a:bound} and the following additional regularity assumptions. 
\begin{assumption} $\EE \in \mathcal{C}^2(\RR^d)$ and there exist $c_1,c_2 >0$ such that
\begin{enumerate}
\item $\displaystyle\sup_{v \in \RR^d} |\nabla \EE(v) | \leq c_1\;;$ 
\item $\displaystyle\sup_{v \in \RR^d} \rev{\|\nabla^2\EE(v)\|_2 }\leq c_2 \;\; \forall\,\, i=1,\ldots,d\, .$
\end{enumerate}
\label{a:reg}
\par
\end{assumption}
Under these assumptions on the objective function $\EE$, the following result holds.

\begin{theorem}\label{maintheo}
Let $f(v,t)$ satisfy the \rev{Boltzmann} equation \rev{\eqref{eq:Boltz}} with initial datum $f_0(v)$ \rev{and binary interaction described by \eqref{eq:binp2}.
Let also Assumptions \ref{a:bound} and \ref{a:reg} hold for $\EE$}. If the model parameters $\{\lambda,\sigma,\beta\}$ \rev{and $f_0(v)$} satisfy 
\begin{align}
\label{t:hp1}
&\mu:= \frac{\lambda}{C_{\beta,\EE}} \rev{- \lambda^2}  - \sigma^2 \kappa > 0 \\
&\nu := \frac{2(\rev{\sqrt{2}}\lambda c_1 + \rev{(\lambda^2+ \sigma^2 \kappa) c_2} )\beta e^{-\beta\underline \EE}}{\mu \|\omega_\beta^\EE\|_{L^1(f_0)}} \rev{ \max \{\sqrt{V(0)}, V(0)\}} < \frac{1}{2} \label{t:hp2}
\end{align}
then there exists $\tilde v \in \RR^d$ such that $m(t) \longrightarrow \tilde{v}$ as $t \rightarrow \infty$. Moreover, it holds the estimate
\be
\EE(\tilde v) \leq \underline \EE + r(\beta) + \frac{\log 2}{\beta}
\ee
where, \rev{if a minimizer $v^\star$ of $\EE$ belongs to $\text{supp}(f_0)$, then} $r(\beta):= -\frac{1}{\beta}\log \|\omega_\beta^\EE\|_{L^1(f_0)} - \underline \EE \longrightarrow 0$ as $\beta \rightarrow \infty$ thanks to the Laplace principle \eqref{Laplace}.
\end{theorem}
\begin{proof}
\rev{Similar to what we did to derive the mean-field scaling limit, we consider the multidimensional Taylor expansion for $\omega_\beta^\EE$
\be
\left\langle \omega_\beta^\EE(v') - \omega_\beta^\EE(v) \right \rangle = \left \langle \nabla \omega_\beta^\EE(v) \cdot (v'-v) +\frac12 (v'-v) \cdot \nabla^2 \omega_\beta^\EE(\hat v)(v'-v) \right \rangle
\ee
where $\hat v = \theta v +(1-\theta)v'$ for some $\theta\in(0,1)$. Thanks to Assumption \ref{a:reg}, one can bound the above terms as 
\be\notag \left \langle \nabla \omega_\beta^\EE(v) \cdot (v'-v) \right \rangle = 
-\beta e^{-\beta \EE(v)} \lambda \nabla \EE(v) \cdot(v_{\beta,\EE}(v,v_*)-v) \geq - \beta e^{-\beta \underline \EE}\lambda c_1|v_{\beta,\EE}(v,v_*) - v|\,.
\ee
By computing the Hessian of $\omega_\beta^\EE(v)$
\[  \nabla^2 \omega_\beta^\EE =  \beta^2 e^{-\beta \EE} \nabla \EE \otimes\nabla  \EE - \beta e^{-\beta \EE} \nabla^2 \EE\,,
\]
we obtain
\begin{align*}
\frac12\left \langle (v'-v)\cdot \nabla^2 \omega_\beta^\EE(\hat v) (v'-v) \right \rangle
& = \left \langle \frac12 \beta^2 e^{-\beta \EE(\hat v)} |\nabla \EE(\hat v) \cdot (v'-v)|^2 - \frac\beta 2 e^{-\beta \EE(\hat v)} (v'-v) \cdot \nabla^2\EE(\hat v) (v'-v) \right \rangle\\
& \geq - \frac\beta 2 e^{-\beta\underline \EE} \| \nabla^2 \EE(\hat v)\|_2  \left \langle|v'-v|^2 \right \rangle \\
&\geq - \frac \beta 2  e^{-\beta\underline \EE}  (\lambda^2 + \sigma^2\kappa) c_2| v_{\beta, \EE}(v,v_*) - v |^2
\end{align*}
where in the last inequality we used Assumption \ref{a:reg} and the fact that 
\[\left \langle |v'-v|^2 \right \rangle \leq 
\lambda^2 |v_{\beta, \EE}(v,v_*) - v|^2 + \sigma^2 \left \langle |D(v,v_*) \xi_1|^2 \right \rangle \leq (\lambda^2 + \sigma^2 \kappa) |v_{\beta, \EE}(v,v_*) - v|^2\,,
\]
by definition of $D(v,v_*)$ and $\xi_1$.}

\rev{We introduce
\be
M_\beta(t) := \int_{\RR^d} \omega_{\beta}^\EE(v) f(v, t) \,dv = \| \omega_\beta^\EE\|_{L^1(f(\cdot,t))}
\ee
and apply the weak formulation \eqref{eq:Boltz} to $\phi(v)= \omega_{\beta}^\EE(v)$ to obtain
\be
\begin{split}
\frac{d M_\beta(t)}{dt} =& \left \langle \int_{\RR^d} \left (\omega_\beta^\EE(v') - \omega_\beta^\EE(v) \right) f(v,t) f(v_*,t) \,dv\, dv_* \right \rangle\\
\geq & - \beta e^{-\beta \underline \EE} \lambda c_1 \int_{\RR^{2d}} |v_{\beta,\EE}(v,v_*) - v|\,f(v,t) f(v_*,t)\, dv\, dv_*\\
& -  \frac \beta 2 e^{-\beta\underline \EE} (\lambda^2 + \sigma^2\kappa) c_2   \int_{\RR^{2d}}  |v_{\beta,\EE}(v,v_*) - v|^2\,f(v,t) f(v_*,t)\, dv\, dv_*\,.
\end{split}
\ee
We recall that
\[
\int_{\RR^{2d}} |v_{\beta,\EE}(v,v_*) - v|^2\,f(v,t) f(v_*,t)\, dv\, dv_* \leq 
\int_{\RR^{2d}} \gamma_{\beta}^\EE(v,v_*)|v_* - v|^2\,f(v,t) f(v_*,t) \,dv\, dv_* = 2 V(t)\,
\]
form which also follows, by Jensen's inequality,
\[ \int_{\RR^{2d}} |v_{\beta,\EE}(v,v_*) - v|\,f(v,t) f(v_*,t)\, dv\, dv_* \leq \sqrt{2V(t)}\,.
\]
Finally, we obtain
\be
\begin{split}
\frac{d M_\beta(t)}{dt} &\geq 
- \beta e^{-\beta \underline \EE}\lambda c_1\sqrt{2V(t)} -  \beta e^{-\beta\underline \EE}  (\lambda^2 + \sigma^2\kappa) c_2 V(t)\, \\
&\geq  - \beta e^{-\beta \underline \EE}  \left(\sqrt{2}\lambda c_1 + (\lambda^2 + \sigma^2 \kappa)c_2 \right) \max\{\sqrt{V(t)}, V(t)\}\,.
\end{split}
\ee
}

\rev{
Now, by definition of $\mu$ it holds $d V(t)/dt \leq - \mu V(t)$ thanks to Proposition \ref{p:locvar}. As we did in the proof of Corollary \ref{c:1}, we apply Grönwall's inequality to obtain an exponential decay of the variance from which follows
\[\rev{ \max \{\sqrt{V(t)}, V(t)\}  \leq  \max \{\sqrt{V(0)}, V(0)\} e^{-\frac12 \mu t}\, \quad \text{for all}\;\; t>0\,.} 
\]}
This leads to a lower bound for \rev{$M_\beta(t)$} in terms of \rev{$M_\beta(0)$}:
\be
\begin{split}
\rev{M_\beta(t)}\geq& \rev{M_\beta(0)} - \beta e^{-\beta\underline \EE} (\rev{\sqrt{2}}\lambda c_1 + \rev{(\lambda^2 + \sigma^2 \kappa)} c_2)  \rev{\max \{\sqrt{V(0)}, V(0)\}} \int_0^t e^{-\frac{1}{2}\mu s} ds \\
\geq&\rev{M_\beta(0)} - \frac{2(\rev{\sqrt{2}}\lambda c_1 + \rev{(\lambda^2 + \sigma^2 \kappa)} c_2)\beta e^{-\beta\underline \EE}}{\mu}  \rev{ \max \{\sqrt{V(0)}, V(0)\}} \\
=& \rev{M_\beta(0)} (1-\nu)\,.
\end{split}
\ee
By definition of $\nu$ and condition \eqref{t:hp2}, it holds
\be
\begin{split}
\rev{M_\beta(t) >}  \frac{1}{2}\rev{M_\beta(0)}\,.
\label{eq:ineq}
\end{split}
\ee
Let us now consider the limit of the above inequality as $t\rightarrow \infty$. 
Since $m(t) \rightarrow \tilde v$ and $V(t) \rightarrow 0$, \rev{it holds}
%We refer to \cite{carrillo2018analytical} for a detailed proof. 
%From the boundedness of $\EE$, this implies that}
\be
\rev{M_\beta(t) = \int \omega_\beta^{\EE} (v) f(v,t)\, dv\; \longrightarrow\; \omega_\beta^\EE(\tilde v) = e^{-\beta \EE(\tilde v)}\quad \text{as} \quad t \to \infty\,.
}\ee 
\rev{The above limit is a consequence of Chebyshev’s inequality, we refer to the proof of \cite[Lemma 4.2]{ carrillo2018analytical} for more details.
Considering the limit of inequality \eqref{eq:ineq} as $t \to \infty$, we have} 
\be
e^{-\beta \EE(\tilde v) >} \frac{1}{2}\rev{M_\beta(0)} \, .
\ee
Finally, we take the logarithm of both sides of the above inequality to obtain
\be
\begin{split}
\EE(\tilde v) &\rev{<} - \frac{1}{\beta}\log\rev{M_\beta(0)} + \frac{\log 2}{\beta} \rev{\,=\,}  \underline \EE + r(\beta) + \frac{\log 2}{\beta}\,,
\end{split}
\ee
where $r(\beta) :=\rev{\,- \frac{1}{\beta}\log M_\beta(0)  - \underline \EE =} - \frac{1}{\beta}\log \|\omega_\beta^\EE\|_{L^1(f_0)}  - \underline \EE $.
\end{proof}

\subsection{\rev{The case with only the macroscopic best estimate}}

\rev{
We now consider the case where only the macroscopic dynamics occurs and the interaction is determined by \eqref{eq:binp3}.
\begin{theorem}\label{t:2}
Let $f(v,t)$ satisfy the \rev{Boltzmann} equation \rev{\eqref{eq:Boltz}} with initial datum $f_0(v)$ \rev{and binary interaction described by \eqref{eq:binp3}.
Let also Assumptions \ref{a:bound} and \ref{a:reg} hold for $\EE$}. If the model parameters $\{\lambda,\sigma,\alpha\}$ \rev{and $f_0(v)$} satisfy 
\begin{align}
\label{t2:hp1}
&\mu:= 2 \lambda - 4 \frac{e^{-\alpha\underline \EE}}{\|\omega_\a^\EE\|_{L^1(f_0)} }  (\lambda^2 + \kappa\sigma^2) > 0 \\
&\nu := \frac{4(2\lambda + \rev{\lambda^2 +} \sigma^2 \kappa)c_2 \alpha e^{-2\a\underline \EE}}{\mu \|\omega_\a^\EE\|_{L^1(f_0)} ^2} V(0) < \frac{3}{4} \label{t2:hp2}
\end{align}
then there exists $\tilde v \in \RR^d$ such that $m(t) \longrightarrow \tilde{v}$ as $t \rightarrow \infty$. Moreover, it holds the estimate
\be
\EE(\tilde v) \leq \underline \EE + r(\a) + \frac{\log 2}{\a}
\ee
where, \rev{if a minimizer $v^\star$ of $\EE$ belongs to $\text{supp}(f_0)$, then} $r(\a):= -\frac{1}{\a}\log \|\omega_\a^\EE\|_{L^1(f_0)} - \underline \EE \longrightarrow 0$ as $\a \rightarrow \infty$ thanks to the Laplace principle \eqref{Laplace}.
\end{theorem}
}

\begin{proof}
\rev{
Similar to the proof of Theorem \ref{maintheo}, we consider the Taylor expansion of $\omega_\alpha^\EE$ which reads as
\be
\left \langle \omega_\alpha^\EE(v') - \omega_\alpha^\EE(v) \right \rangle =\lambda
 \nabla\omega_\alpha^\EE(v) \cdot (v_{\alpha,\EE}(t) - v) + \frac12 \left \langle (v'-v) \cdot \nabla^2\omega_\alpha^\EE(\hat v)(v'-v)         \right \rangle 
 \ee
for some $\hat v \in \RR^d$. As before, by using Assumption \eqref{a:reg} and the definition of $\nabla^2 \omega_{\alpha}^\EE$, the second term can be bounded as
\be
\begin{split}\notag
\frac12 \left \langle (v'-v) \cdot \nabla^2\omega_\alpha^\EE(\hat v)(v'-v)         \right \rangle 
&\geq - \frac\alpha2 e^{- \alpha \EE(\hat v)} (v' - v)\cdot \nabla^2\EE(\hat v) (v' - v) \\
&\geq - \frac \alpha 2 e^{-\alpha \underline \EE} (\lambda^2 + \sigma^2 \kappa) c_2 | v_{\alpha, \EE}(t) - v |^2\,.
\end{split}
\ee
For the first term of the expansion, it holds
\be
\begin{split}\notag
\int_{\RR^{2d}} \lambda \nabla\omega_\alpha^\EE(v) \cdot (v_{\alpha,\EE}&(t) - v) \,f(v,t)f(v_*,t)\,dv\,dv_* \\
& = - \alpha\lambda \int_{\RR^{2d}} e^{-\a\EE(v)} \nabla \EE(v) \cdot (v_{\alpha,\EE}(t) - v) \,f(v,t)f(v_*,t)\,dv\,dv_* \\
& = - \alpha\lambda \int_{\RR^{2d}} e^{-\a\EE(v)} \left (\nabla \EE(v) - \nabla \EE(v_{\alpha,\EE}(t)) \right)\cdot (v_{\alpha,\EE}(t) - v) \,f(v,t)f(v_*,t)\,dv\,dv_*\\
&
\geq - \alpha e^{-\alpha \underline \EE}\lambda c_2 \int_{\RR^{2d}} |v_{\alpha,\EE}(t) - v|^2 \,f(v,t)f(v_*,t)\,dv\,dv_*\,,
\end{split}
\ee
where we used that 
\[\int_{\RR^d} e^{-\a \EE(v)} \nabla\EE(v_{\alpha,\EE}(t)) \cdot (v_{\alpha,\EE}(t) - v)f(v,t)f(v_*,t)\,dv\,dv_* = 0\,.\]
As before, we denote $M_\alpha(t):= \| \omega_\alpha^\EE(t)\|_{L^1(f(\cdot,t))}$. By the weak formulation \eqref{eq:Boltz} it then follows
\be
\begin{split}
\frac{d}{dt}M_\alpha^2(t) &= 2 M_\a(t) \frac{d}{dt}M_\a(t) =2 M_\a(t) \left \langle\int_{\RR^{2d}} \omega_\alpha^\EE(v') - \omega_\alpha^\EE(v) \,f(v,t)f(v_*,t)\,dv\,dv_* \right \rangle \\
& \geq - 4\a c_2 \left(2\lambda + \lambda^2 + \sigma^2\kappa\right) e^{-2\a\underline\EE}  V(t)
\end{split}
\ee
where we used \eqref{eq:valpha-v} to bound the expectation of $|v_{\a, \EE}(t) - v|^2$.
}

\rev{
We now define the time
\be T := \sup \left\{t\,:\, M_\alpha(s) >\frac 12 M_\alpha(0), \; \forall\; s \in[0,t]   \right\} \ee  
and assume that $T< \infty$. By assumption \eqref{t2:hp1} on $\mu$, for all $t \in [0,T]$
\be\notag
2 \lambda - 2 \frac{e^{-\alpha\underline \EE}}{M_\alpha(t)}  (\lambda^2 + \kappa\sigma^2) \geq 2 \lambda - 4 \frac{e^{-\alpha\underline \EE}}{M_\alpha(0)}  (\lambda^2 + \kappa\sigma^2) = \mu >0\,,
\ee
which leads to \[\frac{dV(t)}{dt} \leq - \mu V(t)\] thanks to Proposition \ref{p:macrovar}. Due to Grönwall's inequality one has $V(t) \leq V(0) \exp(-\mu t)$ for all $t \in [0,T]$. 
By assumption \eqref{t2:hp2},
\be\begin{split}
M_\a^2(t) \geq M^2_\a(0) - 4 (2\lambda + \lambda^2 + \sigma^2\kappa) c_2 \alpha
e^{-2\a \underline \EE} V(0) \int_0^t e^{-\mu s} ds \\
> M^2_\a(0) - \frac{4 (2\lambda + \lambda^2 + \sigma^2\kappa)c_2 \alpha
e^{-2\a \underline \EE}}{\mu} V(0) \geq \frac14 M_\a^2(0)
\end{split}\ee 
which implies that for all $t \in [0,T]$, 
\be M_\a(t) > \frac12 M_\a(0)\,.
\label{eq:Mbound}
\ee This means that for some $\delta>0$, $M_\a(t) \geq \frac12 M_\a(0)$ for all $t\in [T, T+\delta)$  which contradicts the definition of $T$. Consequently, $T=\infty$ and hence \eqref{eq:Mbound} holds for all $t>0$. As a consequence, we obtain the exponential decay of the variance 
\be
V(t) \leq V(0) e^{-\mu t} \quad \text{for all} \; t>0\,.
\ee}

\rev{
As we showed in the proof of Corollary \ref{c:1}, there exists a $\tilde v \in \RR^d$ such that $m(t) \to \tilde v$ as $t\to \infty$ with exponential rate, from which follows $M_\a(t) \to e^{-\a \EE(\tilde v)}$. By taking the limit as $t \to \infty$ of \eqref{eq:Mbound}, we obtain
\be e^{-\a\EE(\tilde v)} > \frac12 M_\a(0)
\ee
and we conclude that
\be
\EE(\tilde v) < - \frac1\a \log M_\a(0) + \frac{\log 2}\a = \underline \EE + r(\a) + \frac{\log 2}\a\,,
\ee
where $r(\a) := - \frac{1}{\a}\log M_\a(0)  - \underline \EE = - \frac{1}{\a}\log \|\omega_\a^\EE\|_{L^1(f_0)}  - \underline \EE $.}
\end{proof}

\rev{Finally, it is possible to prove a general convergence result to the global minimum for the case where both the local and global best alignments occur in the particles interaction. In the following, for simplicity we will set $\beta = \alpha$.} 
\rev{
\begin{theorem}\label{t:3}
Let $f(v,t)$ satisfy the \rev{Boltzmann} equation \rev{\eqref{eq:Boltz}} with initial datum $f_0(v)$ \rev{and binary interaction described by \eqref{eq:binp}.
Let also Assumptions \ref{a:bound} and \ref{a:reg} hold for $\EE$}. If the model parameters $\{\lambda_1,\lambda_2,\sigma_1, \sigma_2,\alpha\}$ \rev{and $f_0(v)$} satisfy 
\begin{align}
\label{t3:hp1}
&\mu:= 2 \lambda_2 - 4 \frac{e^{-\alpha\underline \EE}}{\|\omega_\a^\EE\|_{L^1(f_0)} }  (\lambda_2^2 + \kappa\sigma_2^2) - (\lambda^2_1 +\sigma^2_1 \kappa) > 0 \\
&\nu := \frac{8\left(\sqrt{2}\lambda_1 c_1 + (\lambda_1^2 + \sigma_1^2\kappa) c_2+  (2 \lambda_2 + \lambda_2^2 + \sigma_2^2 \kappa)c_2 \right)\alpha e^{-2\a\underline \EE}}{\mu \|\omega_\a^\EE\|_{L^1(f_0)} ^2} \max\{\sqrt{V(0)}, V(0)\} < \frac{3}{4} \label{t3:hp2}
\end{align}
then there exists $\tilde v \in \RR^d$ such that $m(t) \longrightarrow \tilde{v}$ as $t \rightarrow \infty$. Moreover, it holds the estimate
\be
\EE(\tilde v) \leq \underline \EE + r(\a) + \frac{\log 2}{\a}
\ee
where, \rev{if a minimizer $v^\star$ of $\EE$ belongs to $\text{supp}(f_0)$, then} $r(\a):= -\frac{1}{\a}\log \|\omega_\a^\EE\|_{L^1(f_0)} - \underline \EE \longrightarrow 0$ as $\a \rightarrow \infty$ thanks to the Laplace principle \eqref{Laplace}.
\end{theorem}
}
\rev{The proof closely follows the proofs of Theorems \ref{maintheo} and \ref{t:2} and will be omitted for brevity. It is interesting to remark, however, that condition \eqref{t3:hp1} is far less restrictive than the corresponding condition where only the local best is used \eqref{t:hp1}. This suggest to use the local best in practical applications only in combination with the global best.} 

\rev{Before concluding our theoretical analysis, a few remarks are in order.}

\begin{remark}
\begin{itemize}
\rev{\item  The assumptions in Theorems \ref{maintheo}, \ref{t:2} and \ref{t:3} depend strongly on $\beta$ and $\alpha$, which have to be considered as fixed parameters. Therefore, the limits $t\to \infty$ and $\beta$, or $\a \to \infty$ are not interchangeable. Furthermore, for a given $\beta$, or $\alpha$, a choice of $\lambda_1, \lambda_2$ and $\sigma_1, \sigma_2$ satisfying the  assumptions is always possible, at the cost of taking $V(0)$ sufficiently small.
%\rev{\item It is possible to obtain a convergence result also when both the microscopic and the macroscopic best estimate contribute in the particles interaction. This can be done by combining the estimates we derived for the two cases and make assumptions similar to the ones used in Theorem \ref{maintheo} and \ref{t:2}. We omit the details for brevity.}
\item Under the mean field scaling \eqref{eq:scaling}, one can directly derive the equivalent of Theorem \ref{t:3} for the mean-field limit dynamics \eqref{eq:mfb}. For small values of the scaling parameter $\varepsilon$, the quadratic terms in $\lambda_1, \lambda_2$ will vanish and, in the case of global best only, we recover the same convergence result of CBO methods (see for instance \cite[Theorem 3.1] {carrillo2019consensus}).
\item Finally, in the case where the diffusion process in binary interactions is anisotropic, namely \eqref{eq:aniso} holds and therefore $\kappa=1$, convergence to the global minimum is guaranteed with parameter constraints independent of the problem dimensionality. For this reason, in all numerical examples of the next section only anisotropic noise has been considered.}  
\end{itemize}
\end{remark}

%\rev{
%\begin{remark}
%Condition \eqref{t:hp1}, which ensures the variance decay, may become particularly restrictive for large values of $\beta$, especially if compared with the correspondent convergence results obtained when the dynamics is fully macroscopic, see for instance \cite{carrillo2019consensus}. This difference 
%is due to the fact that the fully microscopic dynamics lacks a global alignment process, which is instead present in the CBO dynamics. For this reason, the diffusion parameter needs to be small with respect to $\lambda$ to ensure the variance decay.
%In the numerical simulations, we will see, indeed, that the diffusion parameter associated with the microscopic dynamics, $\sigma_1$, is typically chosen to be smaller then the correspondent macroscopic one, $\sigma_2$.
%\end{remark}
%}
%%%%%%%%%%%%%%%%%%%%%. NUMERICAL RESULTS %%%%%%%%%%%%%%%%%%%%%%%%%

\section{Numerical examples and applications}
\label{S5}

This section is devoted to discuss \rev{the implementation of the proposed methods and to test their performance with the aid of several numerical experiments}. The first \rev{experiment}, \rev{in Section \ref{S52}}, consists of checking the fitness of the macroscopic best estimate in \eqref{ValphaE} employing in the evolution of the dynamic both terms \eqref{WbetaE} and \eqref{ValphaE}, in comparison to the sole presence of one of the two. The second experiment\rev{, presented in \cref{subsec:compSGD},} is devoted to show how even simple 1--dimensional problems may pose serious issues to classical descent methods, whilst the proposed procedure has an high success rate. Finally, the last section presents an application to a classical machine learning problem, showing that KBO methods have the potential to outperform classical approaches. \rev{It should be noted that, in numerical experiments, to facilitate comparison with the literature, we will denote by $f(x)$ the function to be minimized and by $x$ the variable in the search space, instead of $\EE(v)$ and $v$ as in the description of the KBO method.}

\subsection{Implementation}
\label{S51}
The numerical implementation of KBO relies on two different algorithms inspired by Nanbu's and Bird's direct simulation Monte Carlo methods in rarefied gas dynamics \cite{Bird, Nanbu, PRMC}. The former considers at each time step the evolution of distinct pairs of particles, while the latter allows for multiple interactions between pairs of particles in a time step. The methods are summarized in Algorithms \ref{algo:nanbu} and \ref{algo:bird}, the interested reader can find additional details on similar algorithms used in particle swarming in \cite{AlPa, partos13}. \rev{Mathematically, let us remark that in the limit of a large number of particles Nambu's method converges to a discrete-time formulation of \eqref{eq:Boltz}, while Bird's method converges to the continuous-time formulation \eqref{eq:Boltz}.} 

In the algorithms reported, the parameters $\delta_{\mbox{\tiny stall}}$ and $n_{\mbox{\tiny stall}}$ check if consensus has been reached in the last $n_{\mbox{\tiny stall}}$ iterations within a tolerance $\delta_{\mbox{\tiny stall}}$: in such case, the evolution is stopped without reaching the total number of iterations. The initial particles are drawn from a given distribution, typically uniform in the search space unless one has additional informations on the locations of the global minimum. Note that in Bird's algorithm interactions take place without any time counter compared to Nanbu's method. As a consequence the total number of interactions as well as the parameter $n_{\mbox{\tiny stall}}$ have to be adjusted accordingly to the overall number of particles.

\medskip
\begin{algorithm}[H]
\begin{algorithmic}
\footnotesize
\STATE{Input parameters: $N_p$, $N_t$, $\varepsilon>0$, $\sigma_1, \sigma_2$, $\lambda_1, \lambda_2$, $n_{\mbox{\tiny stall}}$ and $\delta_{\mbox{\tiny stall}}$} 
\STATE{Initialise $N_p$ particles: $\{v_i^{(0)}\}_{i=1,\ldots,N_p}$} 
\STATE{$t\gets 0$, $n\gets 0$}
\STATE{Compute $v_{\alpha,\EE}^{(0)}$}
\WHILE{$t<N_t$ and $n<n_{\mbox{\tiny stall}}$}
\FOR{$i=1,\ldots,N_P$}
\STATE{Select uniformly another individual $v_j^{(t)}$, among the others except $v_i^{(t)}$} 
\STATE{Compute $v_{\beta,\EE}\lp v_i^{(t)},v_j^{(t)}\rp$ }
\STATE{$d_{\beta,i} \gets v_{\beta,\EE}\lp v_i^{(t)},v_j^{(t)}\rp-v_i^{(t)}$}
%\STATE{$d_{\beta, j} = v_{\beta,\EE}(v_i^{(t)},v_j^{(t)})-v_j^{(t)}$}
\STATE{$d_{\alpha, i} \gets v_{\alpha,\EE}^{(t)}-v_i^{(t)}$}
%\STATE{$d_{\alpha, j} = v_{\alpha,\EE}-v_j^{(t)}$}
\STATE{Generate $\xi_1, \xi_2 \sim \mc{N}(0,1)$}
\STATE{$v_i^{(t+1)}\gets v_i^{(t)} + \varepsilon\lambda_1 d_{\beta,i}+ {\varepsilon}\lambda_2 d_{\alpha, i} + \sqrt{\varepsilon}\sigma_1\Diag(d_{\beta, i})\xi_1 +  \sqrt{\varepsilon}\sigma_2\Diag(d_{\alpha, i })\xi_2$}
\ENDFOR
\STATE{Compute $v_{\alpha,\EE}^{(t+1)}$}
\IF{$\|v_{\alpha,\EE}^{(t+1)}- v_{\alpha,\EE}^{(t)}\|_2<\delta_{\mbox{\tiny stall}}$}
\STATE{$n\gets n+1$}
\ELSE
\STATE{$n\gets 0$}
\ENDIF
\STATE{$t\gets t+1$}
\ENDWHILE
%\ENDFOR
\end{algorithmic}
\caption{Nanbu KBO}\label{algo:nanbu}
\end{algorithm}

\begin{algorithm}[h]
\begin{algorithmic}
\footnotesize
\STATE{Input parameters: $N_p$, $N_t$, $\varepsilon>0$, $\sigma_1, \sigma_2$, $\lambda_1, \lambda_2$, $n_{\mbox{\tiny stall}}$ and $\delta_{\mbox{\tiny stall}}$} 
\STATE{Initialise $N_p$ particles: $\{v_i^{(0)}\}_{i=1,\ldots,N_p}$} 
\STATE{$s\gets 0$, $n\gets 0$, $N_s \gets N_t N_p/2$, $n_{\mbox{\tiny stall}} \gets n_{\mbox{\tiny stall}}N_p/2$}
\STATE{Compute $v_{\alpha,\EE}^{(0)}$}
\WHILE{$s<N_s$ and $n<n_{\mbox{\tiny stall}}$}
\STATE{Select a random pair $(i,j)$ uniformly among the $\begin{pmatrix}N_p\\2\end{pmatrix}$ possible ones .}
\STATE{Compute $ v_{\beta,\EE}\lp v_i, v_j\rp$ }
\STATE{$d_{\beta,i} \gets  v_{\beta,\EE}\lp v_i, v_j\rp- v_i$, $d_{\beta, j} \gets  v_{\beta,\EE}\lp v_i, v_j\rp- v_j$}
\STATE{$d_{\alpha, i} \gets  v_{\alpha,\EE}^{(s)}- v_i$, $d_{\alpha, j} \gets  v_{\alpha,\EE}^{(s)}- v_j$}
\STATE{Generate $\xi_1, \xi_2, \xi_1^*,\xi_2^* \sim \mc{N}(0,1)$}
\STATE{$ v_i\gets  v_i + \varepsilon\lambda_1 d_{\beta,i}+ \varepsilon\lambda_2 d_{\alpha, i} + \sqrt{\varepsilon}\sigma_1\Diag(d_{\beta, i})\xi_1 +  \sqrt{\varepsilon}\sigma_2\Diag(d_{\alpha, i })\xi_2$}
\STATE{$ v_j\gets  v_j  + \varepsilon\lambda_1 d_{\beta,j}+ \varepsilon\lambda_2 d_{\alpha, j} + \sqrt{\varepsilon}\sigma_1\Diag(d_{\beta,j})\xi_1^* +  \sqrt{\varepsilon}\sigma_2\Diag(d_{\alpha, j })\xi^*_2$}
\STATE{Update $v_{\alpha,\EE}^{(s+1)}$}
\IF{$\|v_{\alpha,\EE}^{(s+1)}- v_{\alpha,\EE}^{(s)}\|_2<\delta_{\mbox{\tiny stall}}$} 
\STATE{$n\gets n+1$}
\ELSE
\STATE{$n\gets 0$}
\ENDIF
\ENDWHILE
\end{algorithmic}
\caption{Bird KBO}\label{algo:bird}
\end{algorithm}

%Furthermore, we denote as KBO$(f,W;s)$ 
%Employing one of the presented algorithms to minimize a function $f$ by using the proposed KBO procedure on the particles' set $W$ and settings $s=\lp\lambda_1, \lambda_2, \sigma_1, \sigma_2, \varepsilon, \diff t, T\rp$ is denoted via KBO$(f,W;s)$.

\begin{figure}[htb]
\newcommand{\factor}{0.4}
\begin{center}
\subfigure[Nanbu KBO, $v_{\beta,\mathcal{E}}$; succ. rate\label{sfig:s_nanbuloc1}]{\includegraphics[width=\factor\textwidth]{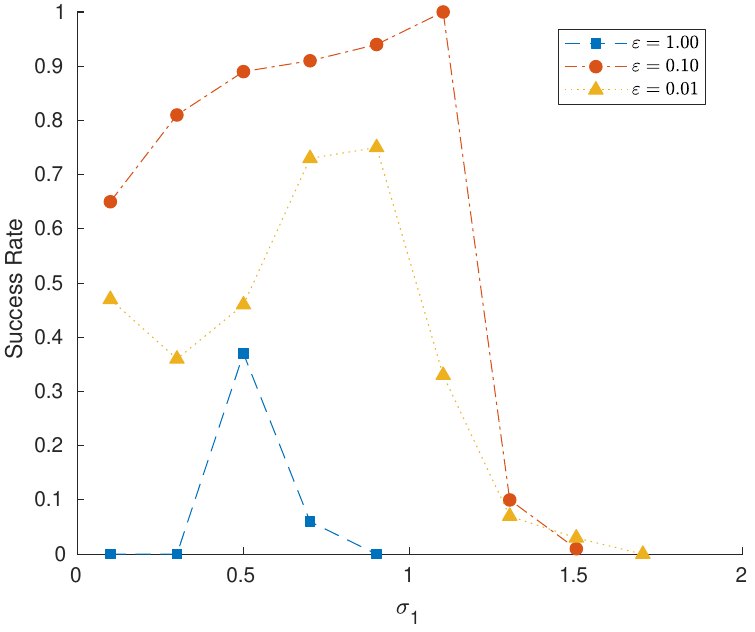}}\hspace{0.1\textwidth}
\subfigure[Nanbu KBO, $v_{\beta,\mathcal{E}}$; iters\label{sfig:i_nanbuloc1}]{\includegraphics[width=\factor\textwidth]{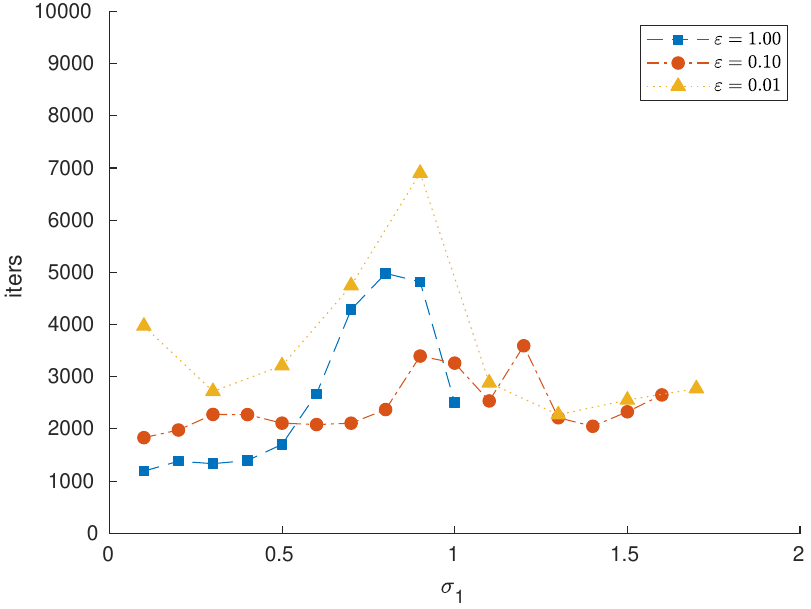}}\\\subfigure[Nanbu KBO, $v_{\alpha,\mathcal{E}}$; succ. rate\label{sfig:s_nanbuloc2}]{\includegraphics[width=\factor\textwidth]{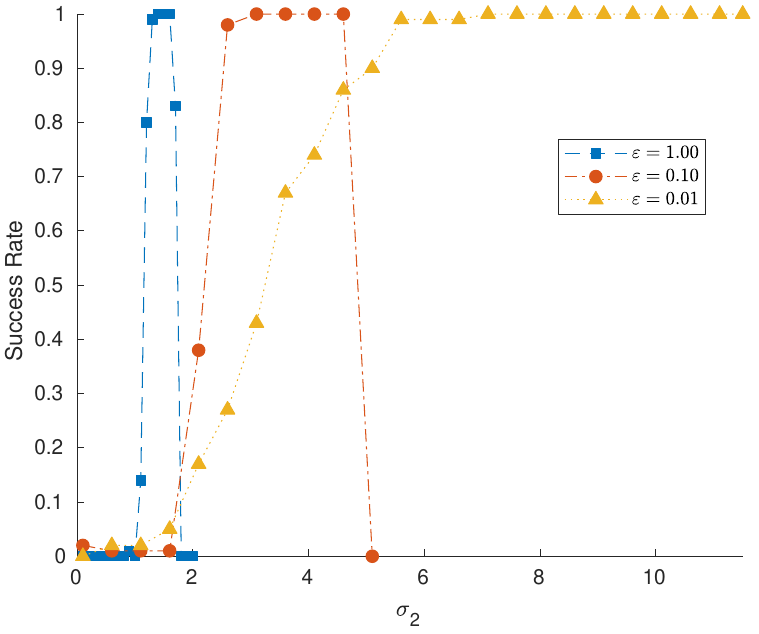}}\hspace{0.1\textwidth}
\subfigure[Nanbu KBO, $v_{\alpha,\mathcal{E}}$; iters\label{sfig:i_nanbuglo2}]{\includegraphics[width=\factor\textwidth]{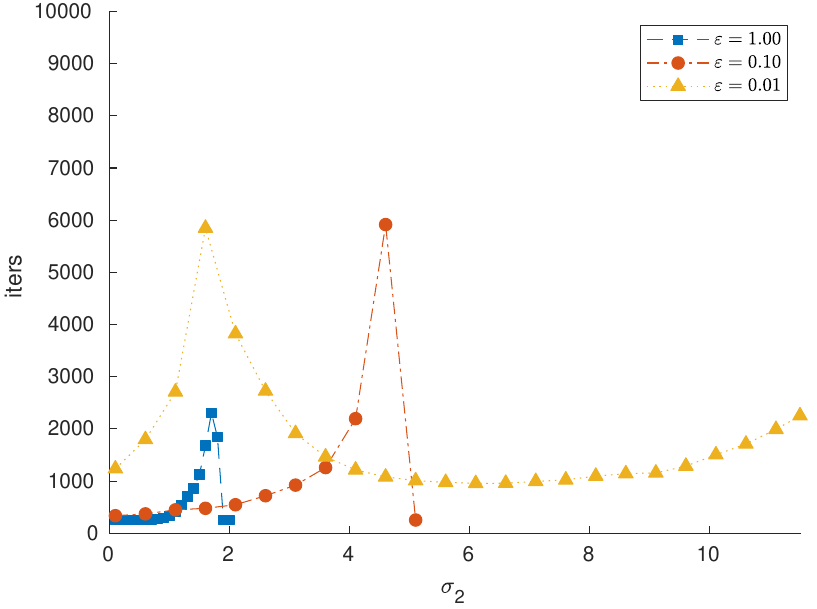}}
\end{center}
\caption{Minimization of Rastrigin function for KBO based on Nanbu's algorithm. From left to right: success rate and average iterations number. Top row refers to the local best only, while the bottom one refers to the global best only.} 
\label{fig:locVSglo1}
\end{figure}

\begin{figure}[htb]
\newcommand{\factor}{0.4}
\begin{center}
\subfigure[Bird KBO, $v_{\beta,\mathcal{E}}$; succ. rate\label{sfig:s_birdloc1}]{\includegraphics[width=\factor\textwidth]{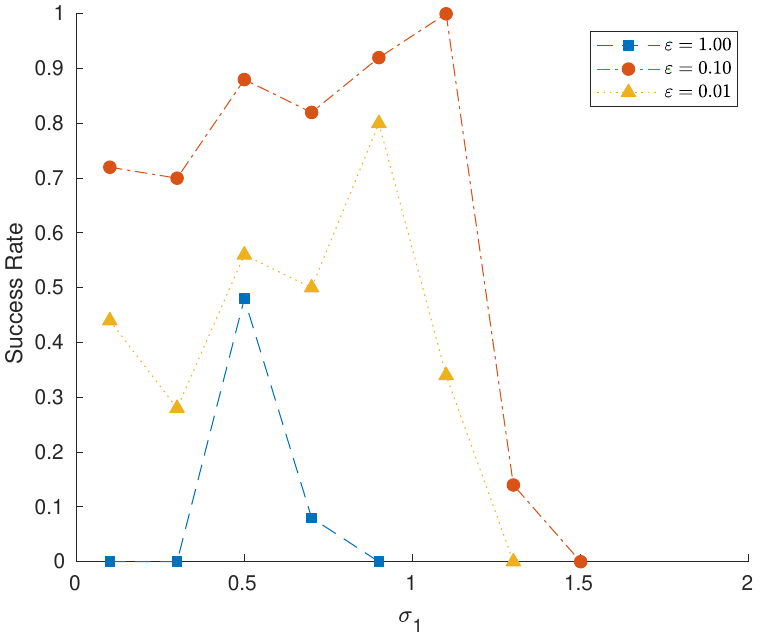}}\hspace{0.1\textwidth}
\subfigure[Bird KBO, $v_{\beta,\mathcal{E}}$; iters\label{sfig:i_birdloc1}]{\includegraphics[width=\factor\textwidth]{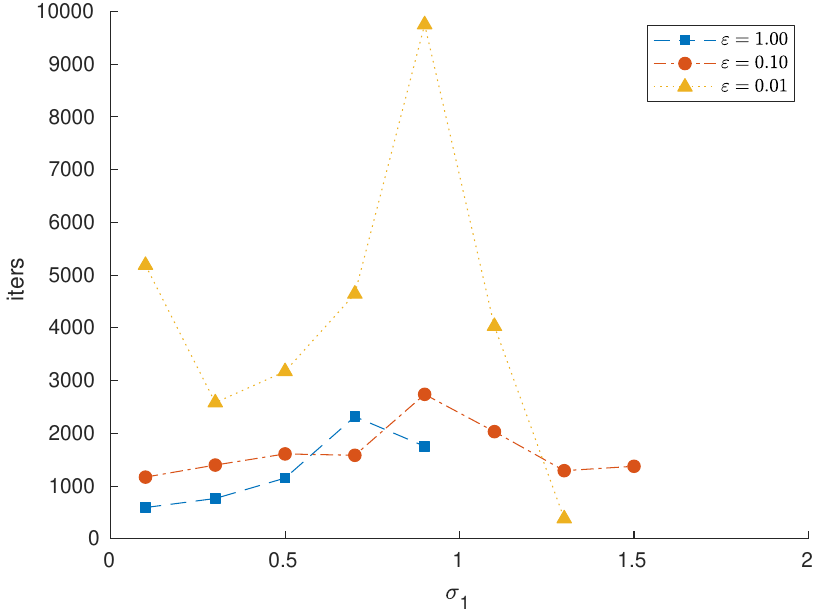}}\\
\subfigure[Bird KBO, $v_{\alpha,\mathcal{E}}$; succ. rate\label{sfig:s_nanbuglo2}]{\includegraphics[width=\factor\textwidth]{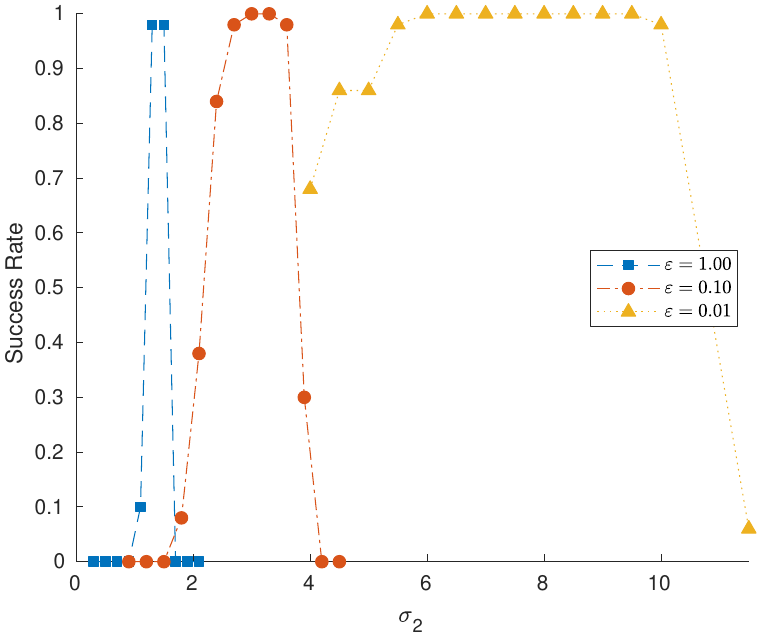}}\hspace{0.1\textwidth}
\subfigure[Bird KBO, $v_{\alpha,\mathcal{E}}$; iters\label{sfig:s_birdglo2}]{\includegraphics[width=\factor\textwidth]{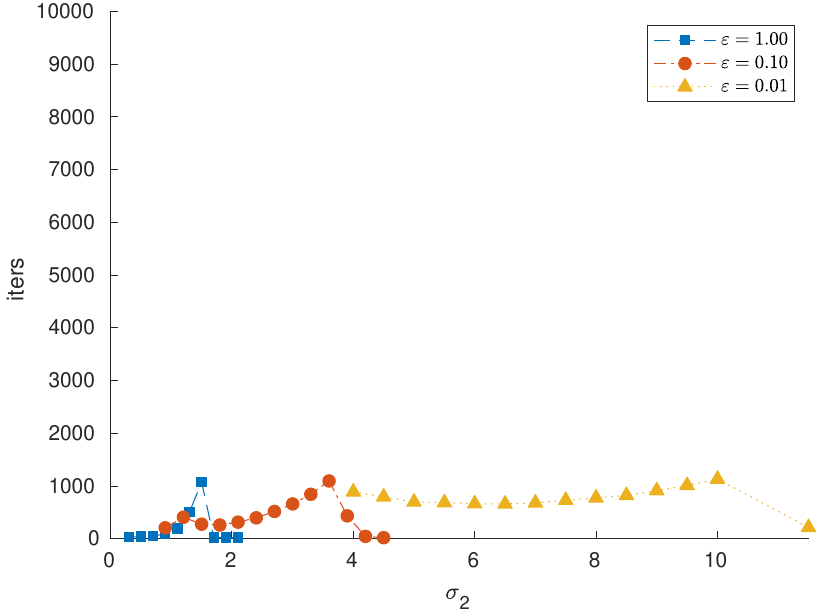}}
\end{center}
\caption{Minimization of Rastrigin function for KBO based on Bird's algorithm. From left to right: success rate  and average iterations number. Top row refers to the local best only, while the bottom one refers to the global best only.} 
\label{fig:locVSglo2}
\end{figure}

\subsection{Validation of the algorithms}
\label{S52}
The validation of the KBO algorithms is pursued initially on a classical benchmark function for global optimization, the Rastrigin function \cite{JamilY13} in dimension $d=20$\rev{, with the global minimum $f(x^\star)=0$, at $x^\star=0$} (see Appendix A). As shown in \cite{carrillo2019consensus, pinnau2017consensus, gp20, fhps20-2}, compared to other benchmark functions the Rastrigin function in high dimension has proven to be quite challenging for CBO-type methods if one is interested in the computation of the precise value $x^\star$ in which the function reaches its global minimum. In fact, the Rastrigin function contains multiple similar minima located in different positions and the minimizer can get easily trapped in one local minimum without being able to compute the global optimum. This test is used to analyze the performances of the two different algorithmic implementations of the method and the effects of the parameters related to the alignment and the exploration processes based on the local best and the global best respectively. 

%\begin{table}[htb]
%\begin{center}
%\begin{tabular}{l|ccrc||ccrc}
%$\varepsilon$ & $N_p$ & $T$ & $n_{\mbox{\tiny stall}}$ & $\delta_{\mbox{\tiny stall}}$& $N_p$ & $T$ & $n_{\mbox{\tiny stall}}$ & $\delta_{\mbox{\tiny stall}}$\\
%\toprule
%&\multicolumn{4}{c||}{Bird Algorithm}&\multicolumn{4}{c}{Nanbu Algorithm}\\
%\midrule
%1 		& 200 & $2\cdot 10^6$ & 50000 & $10^{-4}$ 	& 300 & 10000	& 1000 & $10^{-4}$\\
%0.1 	& 200 & $2\cdot 10^5$ & 50000 & $10^{-4}$	& 300 & 10000	& 1000 & $10^{-4}$\\
%0.01 	& 200 & $2\cdot 10^4$ & 50000 & $10^{-4}$	& 300 & 10000	& 500  & $10^{-4}$\\
%\midrule
%\end{tabular}
%\end{center}
%\caption{Settings for experiments for Bird and Nanbu algorithms. }\label{tab:settings1}
%\end{table}
%
%The settings used for this test are depicted in \cref{tab:settings1}. The main difference among the set--ups of the algorithms consists of the total number of iterations: this is due to the fact  Nanbu's procedure allows the evolution of the whole set of particles, while Bird's algorithm considers the interaction of only two of them. Hence, to have a fair comparison, we have set the maximum number of iterations for Bird's algorithm as $ N_P\,N_{\mbox{\tiny Nanbu}}\varepsilon$, where $N_{\mbox{\tiny Nanbu}}$ is the maximum number of iterations for Nanbu's.

The computational parameters are fixed as $N=200$, $N_t=10000$, $n_{\mbox{\tiny stall}}=1000$, $\delta_{\mbox{\tiny stall}}=10^{-4}$ \rev{and the particles are initially distributed following an uniform distribution in the hypercube $[-3.12, 3.12]^d$, $d=20$}.
Figures \ref{fig:locVSglo1} and \ref{fig:locVSglo2} show the performance of KBO algorithms, considering only  the local best \eqref{ValphaE} or the global best \eqref{WbetaE}. In both figures the first row refers to the case in which only the microscopic estimate has been used, i.e. $\lambda_2=\sigma_2=0$, while $\lambda_1=1$ and $\sigma_1$ ranges in $(0,2]$, while the second row refers to the usage of the sole macroscopic estimation, i.e. $\lambda_1=\sigma_1\rev{=0}$, $\lambda_2=1$ and $\sigma_2\in(0,11]$. Two measures are used for the validation: the first one is the success rate, while the second is the number of iterations. In agreement with \cite{carrillo2019consensus, pinnau2017consensus}, a simulation is considered successful if and only if 
\begin{equation}
\label{eq:SRdef}
\|x^*_\alpha-x^\star\|_\infty <0.25
\end{equation}
where $x^*_\alpha$ is the macroscopic best estimate (provided by \eqref{ValphaE}), while $x^\star$ is the actual minimizer of the Rastrigin function. Note that, in the case where only the local best has been used we still use the global best as an estimate of the global minimizer computed by the algorithm.
The algorithms have been tested for three different choices for $\varepsilon= 1, 0.1$ and $0.01$. Each setting has been  tested for 100 simulations. The local and global minimizers have been evaluated using $\alpha=\beta=5\times 10^6$. For the numerical implementation, we refer to the algorithm introduced in \cite{fhps20-1} which permits to use arbitrary large values of $\alpha$ and $\beta$.

The results for the local best only, in the first row of Figures \ref{fig:locVSglo1} and \ref{fig:locVSglo2} , suggest that there are no great differences in terms of success rate between the two algorithms, even if the choice for $\varepsilon=0.1$ seems to be the best compromise. On the other hand, for $\varepsilon=1$ and $\varepsilon=0.1$ Bird's algorithm needs a slightly less number of iteration for reaching convergence. The second row is devoted to present the results regarding the use of the global best only. In general, decreasing the value for $\varepsilon$ enlarges the interval in which the parameter $\sigma_2$ can be chosen, but at the same time this interval is shifted to the right, meaning that the algorithm needs more noise in order to explore the search domain and identify the global minimum. \rev{It is also clear from Figures \ref{fig:locVSglo1} and \ref{fig:locVSglo2} that the convergence basin with only the local best is significantly smaller than that with only the global best. This is in agreement with the theoretical results of Section 4.}

\begin{figure}[tb]
\begin{center}
\newcommand{\factor}{0.4}
\subfigure[Nanbu KBO, opt $\sigma_2$; succ rate\label{sfig:botha}]{\includegraphics[width=\factor\textwidth]{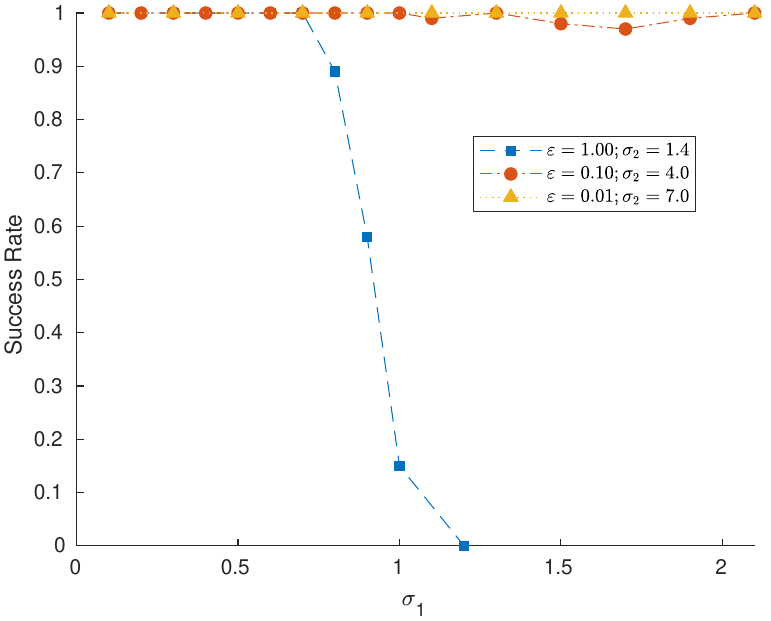}}\hspace{0.1\textwidth}
\subfigure[Nanbu KBO, opt $\sigma_2$; iters]{\includegraphics[width=\factor\textwidth]{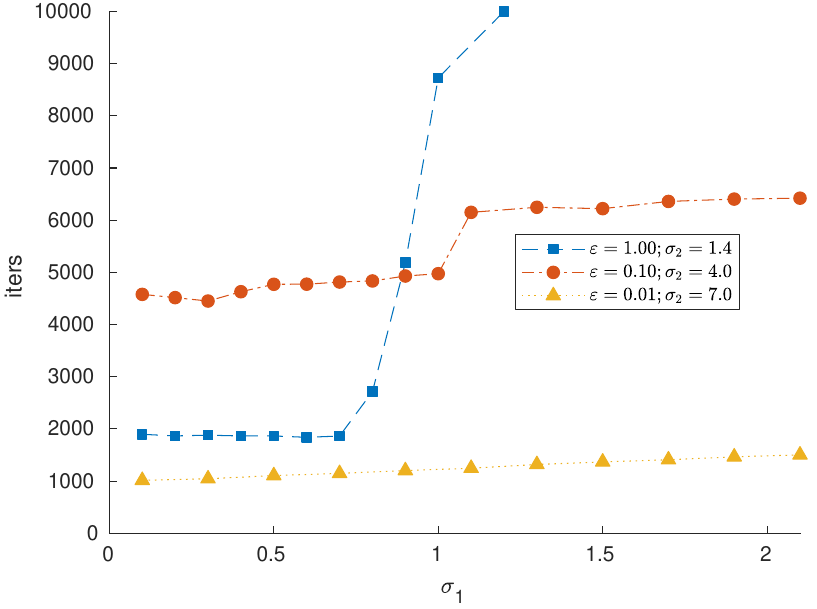}}\\
\subfigure[Nanbu KBO, opt $\sigma_1$; succ rate]{\includegraphics[width=\factor\textwidth]{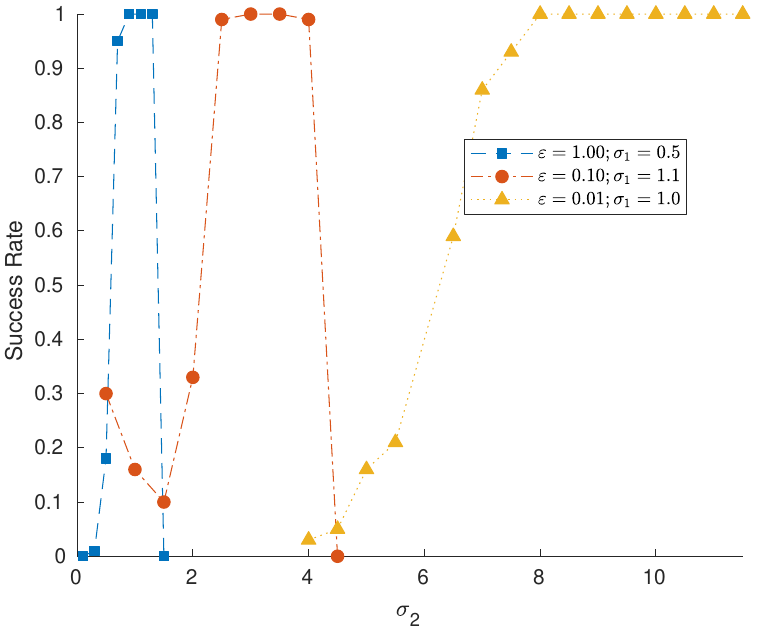}}\hspace{0.1\textwidth}
\subfigure[Nanbu KBO, opt $\sigma_2$; iters]{\includegraphics[width=\factor\textwidth]{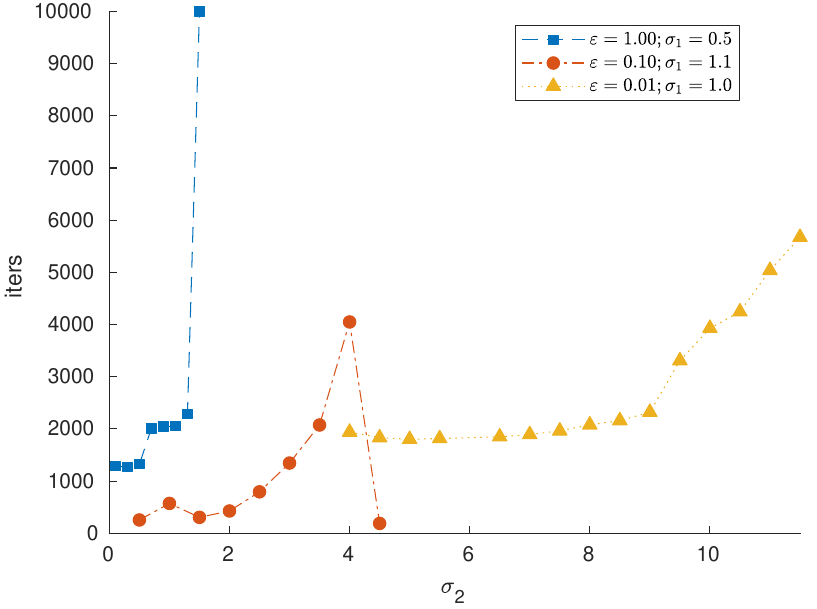}}
\end{center}
\caption{Minimization of Rastrigin function for KBO based on Nanbu's algorithm. From left to right: success rate and average iterations number using both local and global best. Top row refers to the optimal value for the global best, while the bottom one refers to the optimal value for the local best. }
\label{fig:both1}
\end{figure}

\begin{figure}[h]
\begin{center}
\newcommand{\factor}{0.4}
\subfigure[Bird KBO, opt $\sigma_2$; succ rate]{\includegraphics[width=\factor\textwidth]{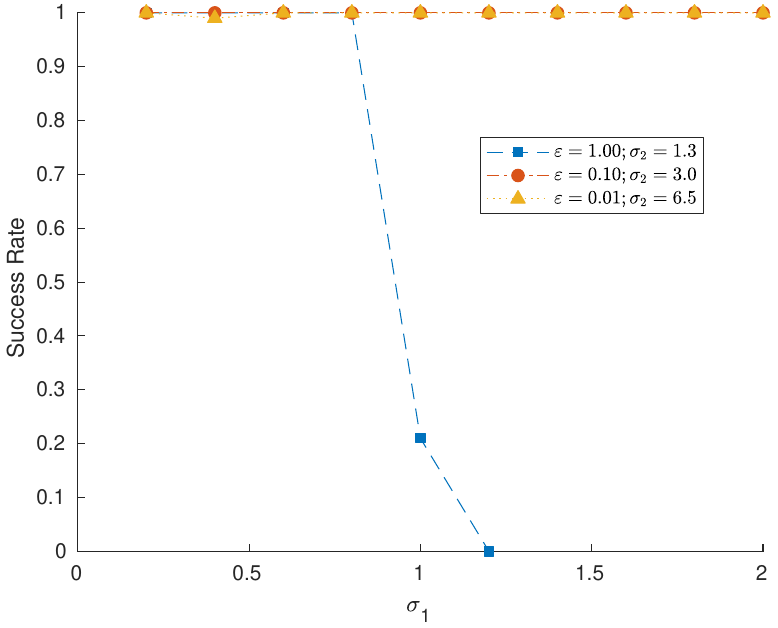}}\hspace{0.1\textwidth}
\subfigure[Bird KBO, opt $\sigma_2$; iters]{\includegraphics[width=\factor\textwidth]{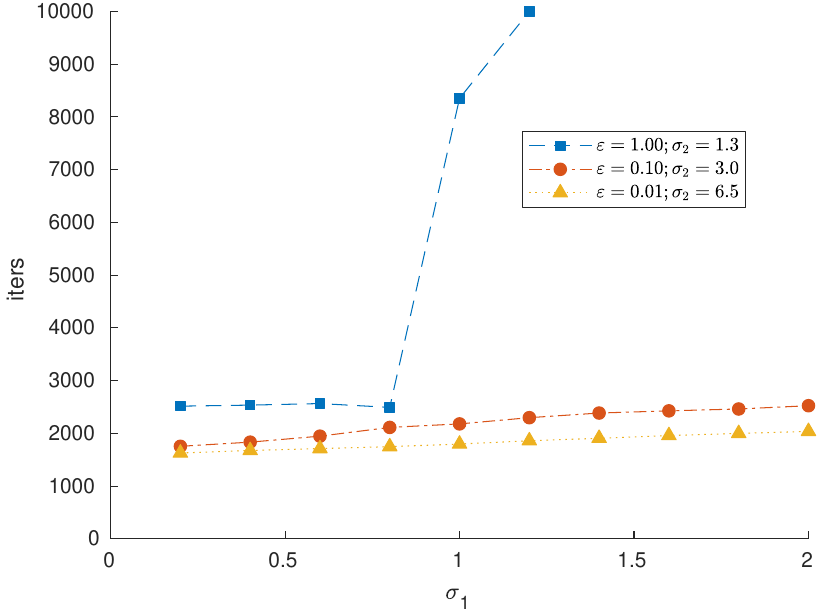}}\\
\subfigure[Bird KBO, opt $\sigma_1$; succ rate]{\includegraphics[width=\factor\textwidth]{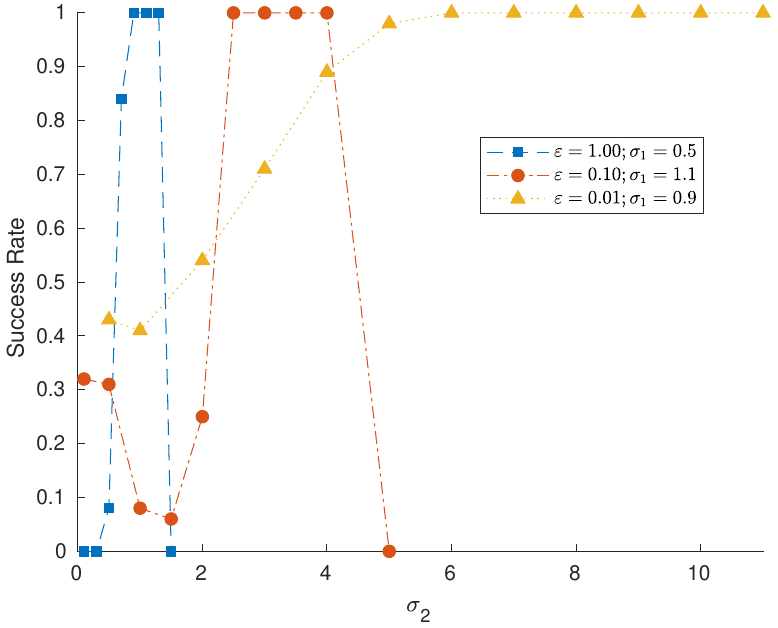}}\hspace{0.1\textwidth}
\subfigure[Bird KBO, opt $\sigma_1$; iters]{\includegraphics[width=\factor\textwidth]{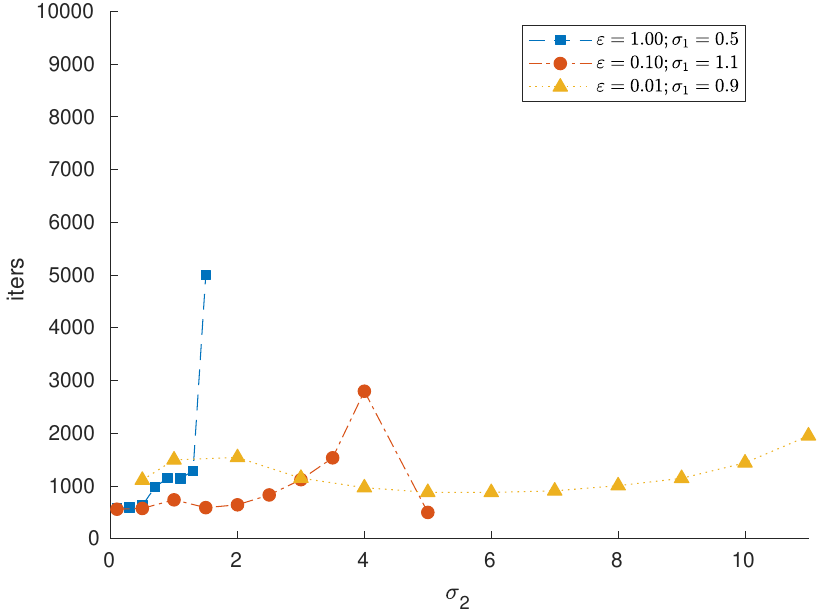}}
\end{center}
\caption{Minimization of Rastrigin function for KBO based on Bird's algorithm. From left to right: success rate and average iterations number using both local and global best. Top row refers to the optimal value for the global best, while the bottom one refers to the optimal value for the local best. }
\label{fig:both2}
\end{figure}

Note that the convergence region for Nanbu's algorithm is slightly wider and that, as in the previous case, the Bird algorithm needs a lower number of iteration to reach convergence. 

\cref{fig:both1,fig:both2} refer to the case in which both microscopic and macroscopic estimates are used in the procedure. In the first row, $\sigma_2$ has been chosen as the optimal value that provided the best success rate in the previous experiment, in the second row the same strategy is applied to $\sigma_1$. The depicted plot show that the performance drastically improves for certain options (check in particular \cref{sfig:botha}) and the required iteration number is decreasing too. \rev{This result is also in agreement with the theoretical analysis at the end of Section \ref{S4} that indicates an increase of the basin of convergence of the method based on the microscopic best when used in combination with the macroscopic best.}
As a final comment we can mention that Bird's algorithm, thanks to the multiple interactions, produced less fluctuations in the numerical solution compared to Nanbu's algorithm. This is well known in rarefied gas dynamics where the algorithms have their origins \cite{PRMC}. In our specific case, this translates is slightly narrower convergence regions and slightly faster convergence rates.

\subsection{Comparison with Stochastic Gradient Descent}
\label{subsec:compSGD}
Next, we considered a test case to compare the proposed KBO algorithms with the classical Stochastic Gradient Descent (SGD). While the main interest in a gradient-free method is in situations where gradient computation is either not possible or is particularly expensive, the purpose of this simple numerical test, originally introduced \rev{in} \cite{carrillo2019consensus} is to illustrate the potential advantages of a consensus-based method even in circumstances where the gradient is available but get easily trapped into local minima without allowing the identification of the global minimum.   

Following \cite{carrillo2019consensus}, we want to minimize the function
\begin{equation}
\label{eq:mintest}
L(x) = \frac1n\sum_{i=1}^n f(x,\xi_i) 
\end{equation}
where
$$
f(x,\xi_i)  = \exp\lp\sin(2x^2)\rp + \frac{1}{10}\lp x-\xi_i-\frac{\pi}{2}\rp^2, \quad\xi_i\sim\mathcal{N}(0,0.01)
$$

The plot of \eqref{eq:mintest} together with its minimum $f(x^\star)$ at $x^\star=1.5353$ (with $n=10000$) is shown in \cref{fig:mintest1}. 

\begin{figure}[tb]
\begin{center}
\includegraphics[width=0.45\textwidth]{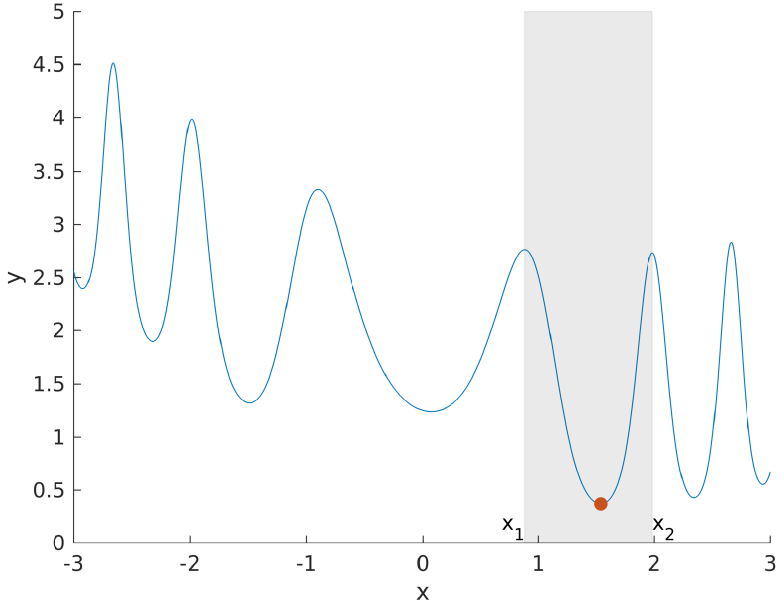}
\end{center}
\caption{Plot of \eqref{eq:mintest}. The orange dot refers to the minimum of the function, the shaded area to the basin of attraction for SGD, and $x_1$ and $x_2$ to the position of the peaks of the basin.}\label{fig:mintest1}
\end{figure}

The SGD procedure is shown in Algorithm \ref{al:SGD1}: this algorithm implements the idea of minibatches, which consists of dividing the set $\{\xi_i\}_{i=1,\ldots,n}$ (the equivalent of a training set in Machine Learning problems) in smaller  $n/m$ subsets where $m$ is the size of each subset,  and then use the descent direction given by the average of these $m$ gradients computed at the current iterate. Exploring the whole set $\{\xi_i\}_{i=1,\ldots,n}$ is called an \emph{epoch} and one can decide to iterate the procedure for several epochs. The parameter $\gamma$ chosen in Algorithm \ref{al:SGD1} is the stepsize, called \emph{learning rate} in Machine Learning framework.

\begin{algorithm}[tb]
\begin{algorithmic}
\footnotesize
\STATE{Choose the learning rate $\gamma$, the Batch Size $m$, the number of Epochs $E$ and the tolerance $\epsilon$.} 
\STATE{Set $e=0$, $k=0$; generate $x^0\sim\mathcal{U}(-3,3)$. Set the number of iterations per epoch $I = \frac{n}{m}$.}
\WHILE{$e<E$ and $|\nabla L(x_k)|>\epsilon$}
\STATE{$i\gets 1$}
\WHILE{$i\leq I$ and $|\nabla L(x_k)|>\epsilon$}
\STATE{\vskip .2cm $\displaystyle x^{k+1} = x^k - \frac{\gamma}{m}\sum_{\ell\in b_k}\nabla f(x^k,\xi_\ell)$}
\STATE{where $b_k$ is a random index set drawn from $\{1,\ldots,n\}$ of size $m$.} 
\STATE{$k\gets k+1$}
\ENDWHILE
\ENDWHILE
%\ENDFOR
\end{algorithmic}
\caption{SGD for minimizing \eqref{eq:mintest}}\label{al:SGD1}
\end{algorithm}

We minimize the function \rev{given} in \eqref{eq:mintest} with $n=10000$ by using both SGD and the proposed KBO algorithm: the former is set with $\gamma=0.1, m=100$, number of epochs equal to one and the procedure is stop\rev{ped} when $|\nabla f(x^k)|<\epsilon$, with $\epsilon=0.01$, while the setting for KBO can be found in \cref{tab:resSGD}, additional \rev{parameters} are $\delta_{\mbox{\tiny stall}}=10^{-4}$. For SGD the starting point is uniformly chosen in $[-3,3]$, the initial 20 particles are chosen in the same interval for KBO. We run 1000 simulations for SGD and 50 simulations for KBO: this is due to the equivalence of 20 runs of SGD to one of KBO. Indeed, the former case is equivalent to consider 20 different particles and then the minimization of the function is pursued independently on each particle. A simulation is considered successful for SGD if and only if the final iterate $x_\alpha^*$ satisfies $|x_\alpha^*-x^\star|<0.25$; for each simulation of KBO we count how many particles (in percentage) lie in the open ball $\mathcal{B}_{0.25}(x^\star)$, i.e. how many particles reached a consensus around the actual solution: \cref{tab:resSGD} collects the average of this consensus among the simulations.

\begin{table}[htb]
\begin{center}
\small
\begin{tabular}{l|l|c|c|c|c}
\small
Method & $\varepsilon$ & $\sigma_1$ & $\sigma_2$ & $n_{stall}$ & Success Rate\\
\midrule
SGD &* &*&*&*&18.00\% \\
\midrule
\multicolumn{6}{c}{Nanbu algorithm}\\
\midrule
KBO  & 1 & 0.1 & 0.5 & 50& 98.50\%  \\
KBO  & 0.1 & 1 & 1 & 50& 100.00\%  \\
%KBO   & 0.1 & 1 & 1 & 10 & 97.60 &  28.21 \\
KBO & 0.01 & 1 & 5 & 50& 98.15\%  \\
\midrule
\multicolumn{6}{c}{Bird algorithm}\\
\midrule
KBO  & 1 	& 0.5 & 0.5 & 50	& 98.50\%  \\
KBO  & 0.1 	& 1.0 & 1.3	& 50	& 100.00\%  \\
KBO  & 0.01	& 1.0 & 6.5	& 50	& 98.70\%  \\
\end{tabular}
\end{center}
\caption{Performances of SGD and KBO. In KBO algorithms we fixed the number of particles $N_p=20$ and the maximum iterations number $N_t=100$.}\label{tab:resSGD}
\end{table}

As shown, for this test case KBO algorithms outperforms the SGD method: even for a small number of particles ($N_p=20$), the minimum of the function is well recovered. The success rate of SGD is not surprisingly low: indeed, being a descent method without momentum, it hugely suffers from the presence of many local minima and from the initial position. The success rate of 18\% is very close to the probability of randomly choosing the initial iterate in the interval containing the actual minimum, shaded in gray in \cref{fig:mintest1}: $|x_2-x_1|/6=0.1833$. Enlarging or reducing the interval in which the initial point is chosen increases or decreases accordingly the success rate of SGD, while KBO does not seem to suffer from this problem. In conclusion, we observe how the implementation via Nanbu's method leads to a higher success rate and how in general Bird's method requires larger $\sigma_1$ and $\sigma_2$ exploration parameters. The latter aspect is in agreement with the lower statistical fluctuation of Bird's method and has been already observed in the previous test case, an aspect that is advantageous in the simulation of physical particles in the context of rarefied gas dynamics, but can prove counterproductive in the case of minimum search problems. For this reason, in the following, we will limit the presentation of subsequent numerical tests to the use of Nanbu's algorithm.

\subsection{Results on high dimensional benchmark functions}

This section is devoted to test the performance of the KBO approach on classical benchmark functions in a high dimensional framework ($d=50$). The related optimization problems have been solved by using a common set of parameters for KBO algorithm
\begin{equation} 
\lambda_1=\lambda_2=1,\quad \sigma_1=.1,\quad \sigma_2=6,\quad \epsilon=0.01,\quad n_{stall}=500,\quad \delta_{stall} = 10^{-4}
\label{eq:params}
\end{equation} 
and the maximum number of iteration is fixed to 10000. The numerical implementation of KBO approach relies on Algorithm \ref{algo:nanbu}. \cref{tab:bench} presents the results obtained on the functions listed in Appendix A. 

\cref{tab:bench} presents the success rate defined as in \eqref{eq:SRdef}
$$
\rev{\|x_\alpha^*-x^\star\|_\infty\leq \delta},
$$
\rev{where $\delta$ controls the severity of the criterion. We chose two different values, namely $0.25$ and $0.1$.} \rev{We computed also the average number of iteration for achieving convergence}. These results are obtained via 100 runs of each instance of the optimization problems. Two further performance measures are reported, the former being the expected error in Euclidean norm, defined as
$
\mathbb{E}\left[ | x^*_\alpha-x^\star|\right]
$
where $x^\star$ is the solution and $x^*_\alpha$ is the global estimate given by KBO procedure achieved for a successful run. The other measurement is the function value obtained at $x^*_\alpha$.

\begin{table}[htbp]
\begin{center}
\begin{tabular}{ll|c|c||ll|c|c}
Function& & $\delta=0.25$ & $\delta=0.01$&Function&& $\delta=0.25$ & $\delta=0.01$ \\
\toprule
\cmidrule{3-4}
\multirow{5}{*}{Salomon}			& SR 	 	& 100\% 	& 100\%			 		&	\multirow{5}{*}{Rastrigin}			& SR 	 	& 75\%		& \rev{84\%} 		\\
							 		& Iters  	& 6306		& \rev{10000} 	  		&										& Iters  	& 3893 		& \rev{2320}		\\
							 		& Error  	& 9.64e-02  & \rev{4.92e-02}		&										& Error  	& 6.91e-01 	& \rev{2.23e-05}	\\
							 		& Fval   	& 0.96		& \rev{0.49}	 		& 										& Fval   	& 0.25 		& \rev{8.95e-7}	 	\\ 
							 		& $N_a$   	& 133		& \rev{215}		 		&										& $N_a$   	& 182 		& \rev{804}			\\
\midrule
\multirow{5}{*}{Griewank}			& SR 	 	& 100\%		&  100\%			 	& \multirow{5}{*}{Schwefel 2.22}		& SR 		& 100\%		& 100\%				\\
							 		& Iters  	& 2722		&  \rev{1696} 	  		&							 			& Iters  	& 2165		& \rev{1631}		\\
							 		& Error  	& 9.22e-03	&  \rev{7.29e-03}	 	&							 			& Error  	& 1.27e-03	& \rev{1.49e-06} 	\\
							 		& Fval   	& 2.49e-2	&  \rev{1.04e-2}		&							 			& Fval   	& 0.27 & \rev{6.9e-4}	 		\\
							 		& $N_a$   	& 258		&  \rev{985}			&							 			& $N_a$   	& 335	& \rev{1017}			\\

\midrule
							 			
\multirow{5}{*}{StyLank}			& SR 	 	& 77\%		&  \rev{100\%} 			&	\multirow{5}{*}{Schwefel 2.23}		& SR 	 	& 100\%		& 100\%			 	\\
									& Iters  	& 5923		&  \rev{2062} 	  		&										& Iters  	& 10000		& 10000		 	  	\\
							 		& Error  	& 4.56e-03	&  \rev{4.70e-05} 		&										& Error  	& 4.53e-02	& \rev{4.69e-02}	\\
							 		& Fval   	& -1958.29	&  -1958.29	 			&									 	& Fval   	& 1e-5		& \rev{3.74e-8}		\\ 
							 		& $N_a$   	& 132		&  \rev{874}		 	&										& $N_a$   	& 75		& \rev{215}			\\
\midrule
\multirow{5}{*}{Neg. Exp.}			& SR 	 	& 100\%		&  100\%			 	& 	\multirow{5}{*}{Sphere}				& SR 	 	& 100\%		& 100\%		 		\\
							 		& Iters  	& 2517		&  \rev{1325}			& 										& Iters  	& 2368 		& \rev{1529}		\\
							 		& Error  	& 1.11e-03	&  \rev{1.40e-03} 		&										& Error  	& 1.02e-03	& \rev{1.88e-04}	\\
							 		& Fval   	& -1			&  -1		 		& 										& Fval   	& 1.00e-5 	& \rev{9.35e-7}	 	\\ 
							 		& $N_a$   	& 271			&  \rev{1129} 		& 										& $N_a$   	& 291 		& \rev{1051}		\\
\midrule
\multirow{5}{*}{Sum of Square}		& SR 	 	& 100\%		&  100\%			 	& 	\multirow{5}{*}{Ackley}				& SR 	 	& 100\%		& 100\%			 	\\
							 		& Iters  	& 2788		&  \rev{1719} 	  		& 										& Iters  	& 2701		& \rev{1674} 	  	\\
							 		& Error  	& 1.15e-03	&  \rev{2.96e-05} 		&										& Error  	& 1.69e-03	& \rev{3.87e-06}	\\
							 		& Fval   	& 2.93e-3	&  \rev{1.02e-6}	 	& 										& Fval   	& 3.32e-2	& \rev{7.00e-5}	 	\\ 
							 		& $N_a$   	& 252		&  \rev{966}		 	& 										& $N_a$   	& 259		& \rev{994}			\\
\bottomrule
\end{tabular}
\end{center}
\caption{Performance of KBO on benchmark functions in dimension $d=50$. All tests were run with the same parameters setting \eqref{eq:params}, and the initial position of the particles are chosen uniformly. Each instance was run for 100 times starting with $N_p=2000$ particles. The table reports the success rate (SR), the average number of iteration (Iters), the mean square error (Error) and the the average functions values (Fval) achieved on successful runs, and the \rev{arithmetic} average number of particles ($N_a$) used along the simulation.}
\label{tab:bench}
\end{table}

We employed here a strategy to dynamically reduce the number of particles used in the procedure. Indeed, as observed in \cite{fhps20-2}, a constant number of particles is not optimal: while the dynamic evolves, the variance of the system diminishes due to consensus. We may then reduce the number of particles, according to this variance decreasing, using the following strategy: compute the variance $S_t$ of the system at time $t$
$$
S_t = \frac{1}{N_t}\sum_{i=1}^{N_t}\rev{|}v_i^{(t)} - \bar{v}\rev{|}^2, \quad \bar{v}=\frac{1}{N_t}\sum_{i=1}^{N_t}v_i^{(t)}
$$ 
where $N_t$ is the number of particles at time $t$. As the consensus increases, the variance decreases: $S_{t+1}\leq S_t$, then the number of particles can be decreased following the ratio $S_t/S_{t+1}\leq 1$, using the formula
\begin{equation}
\label{eq:npartdec}
N_{t+1} = \left\llbracket N_t\lp1+\mu \lp\frac{\hat{S}_{t+1}-S_t}{S_t}\rp\rp \right\rrbracket
\end{equation}
with $\mu\in[0,1]$, $\llbracket x\rrbracket$ denoting the integer part of $x$ and 
$$
\hat{S}_{t+1} = \frac{1}{N_t}\sum_{i=1}^{N_t}\rev{|}v_i^{(t+1)} - \hat{v}|^2, \quad \hat{v}=\frac{1}{N_t}\sum_{i=1}^{N_t}v_i^{(t+1)}.
$$
For $\mu=0$ the discarding procedure is not employed, while for $\mu=1$ the maximum speed up is achieved. For $\mu>0$, a minimum number of particles $N_{\min}$ is set and the reducing procedure is adopted every $t_r$ iterations. For more practical detail, the interested reader may refer to \cite{fhps20-2}. In the experiments presented in \cref{tab:bench}, \rev{we set $\mu=0.1$ for $\delta=0.25$ and $\mu=0.03$ for $\delta=0.1$}, $t_r=10$ and $N_{\min}=10$. 

The initial distribution of the particles is uniform in the cube $[-1,1]^d$, while the initial number of particles is set to 2000. A rescaling strategy is adopted for the dynamics evolution: before computing the function values, the particles are rescaled into the benchmark research domain. For example, in the case of the Griewank function initially the candidates are uniformly drawn from $[-1,1]^d$: to compute the function values in these candidates the latter are rescaled into $[-600,600]^d$ and then these values are used in successive computation of $v_\alpha$ and $v_\beta$. 
%This strategy improves the success rate of the method in a remarkable way.

\cref{tab:bench} shows that the success rate is very high and the error is very low for almost of the benchmark functions. The average number of particles decreases, reaching one tenth of the initial number in some cases, reducing overall both computational cost and time. \rev{Nonetheless, lowering the parameter $\mu$ induces a higher success rate even with a more strict criterion ($\delta=0.1$): this amounts to use a larger number of particles, but at the same time it lowers the number of iterations in most cases. The trade--off to be considered is between computational time and computational cost: this consideration should be done case by case, since it  depends on the function to minimize.} 
%Further numerical experiments (not reported here) showed that the initial number of particles can be set also to 500 for several functions (Ackley, Salomon).

\subsection{\rev{Application} to a machine learning problem}

In the last test case, we apply the KBO technique to a classical problem of Machine Learning: the scope is to recognize  digital numbers contained in images of the MNIST data set, by using a shallow network
$$
f(x;W,b) = {\rm softmax} \lp{\rm ReLU} \lp Wx+b\rp\rp
$$
where $x\in\RR^{784}, W\in\RR^{10\times 784}$, $b\in\RR^{10}$. Moreover
$$
{\rm softmax}(x) = \frac{e^x_i}{\sum_i e^x_i}\,, \quad {\rm ReLU(x) = \max(0,x)}
$$
\AB{being ReLU the well--known Rectified Linear Unit function}. The training of the shallow network consists in minimizing the following function
$$
L(X,y;f)=\frac{1}{n}\sum_{i=1}^n \ell\lp f(X^{(i)};W,b), y^{i}\rp, \quad \ell(x,y) = -\sum_{i=1}^{10}y_i\log(x_i)
$$ 
where $X$ is the training dataset, whose images are vectorized ($\mathbb{R}^{28\times 28}\to \mathbb{R}^{784}$) and stacked column--wise. The function $\ell$ is the cross entropy. 

We adopt a minibatch strategy both for the training set and for the particles used in KBO. The former consists in the classical strategy, depicted also in Algorithm \ref{al:SGD1}, while the latter divides the particles set in $N_p/\AB{m_p}$ minibatches, where $N_p$ is the number of total particles and $\AB{m_p}$ is the number of particles in each batch. The KBO procedure is then iterated on the training batches. The final strategy is depicted in Algorithm \ref{al:MNIST}.
\begin{algorithm}[htb]
\begin{algorithmic}
\footnotesize
\STATE{\textit{Training Set and Labels}: $X\in\mathbb{R}^{784\times n}, y\in\mathbb{R}^{10\times n}$. Sets the number of epochs $E$ and the batchsize $m_t$.}
\STATE{\textit{Setting for KBO}: set $s = \lp\sigma_1, \sigma_2, \lambda_1, \lambda_2, \varepsilon, \alpha,\beta, T, \diff t = \varepsilon\rp$.}
\STATE{\textit{Initial candidates}: $W\in \mathbb{R}^{7840\times N_p}, b\in\RR^{10}$. Select the particles' batch size \AB{$m_p$}}
\STATE{Set $M=n/m_t$, $P=\AB{N_p/m_p}$.}
\FOR{$e=1,\ldots,E$}
\STATE{Reorganize the training set in $M$ batches: $B_1, B_2, \ldots, B_M$}
\FOR{$m=1,\ldots,M$}
\STATE{Reorganize the particles set in $P$ batches: \AB{$\mathcal{B}_1, \mathcal{B}_2, \ldots, \mathcal{B}_P$}}
\FOR{$k=1,\ldots,P$}
\STATE{$W_{\AB{\mathcal{B}}_k}, b_{\AB{\mathcal{B}}_k} \gets \mbox{KBO}(L(X_{B_m},y_{B_m};f), W_{\AB{\mathcal{B}}_k},b_{\AB{\mathcal{B}}}; s)$}
\ENDFOR
\ENDFOR
\ENDFOR
%\ENDFOR
\end{algorithmic}
\caption{Nanbu KBO for ReLU network}\label{al:MNIST}
\end{algorithm}

At each epoch, the training dataset is shuffled in order to have different elements inside the batches. When exploring the current training batch, the particles are shuffled too. For our experiment, we used a dataset \footnote{\url{http://yann.lecun.com/exdb/mnist/}} with 10000 images, 1000 per class, for the training and 10000 images, 1000 per class, for validation. We compared the SGD method and KBO, both set with 20 epochs and minibatch size of 128; all the images in the training set have been normalized \AB{via zero centering and dividing by the standard deviation computed among the entire dataset}. The learning rate for SGD is set to $\gamma=0.1$, without momentum, with starting point randomly selected via a Gaussian distribution of zero mean and unitary variance. The settings for KBO is given by $\sigma_1 =\sigma_2=1, \lambda_1=\lambda_2=1$, $\varepsilon=\diff t = 0.1, \alpha=\beta=5\cdot 10^6$ and we selected $\AB{m_p=5}$ batches and $N_p = 500$ particles. The initial candidates are  randomly picked from a Gaussian Distribution with zero mean and unitary variance.

\begin{figure}[htb]
\begin{center}
\subfigure[KBO without particle reduction.\label{sfig:KOMnored}]{\includegraphics[width=0.5\textwidth]{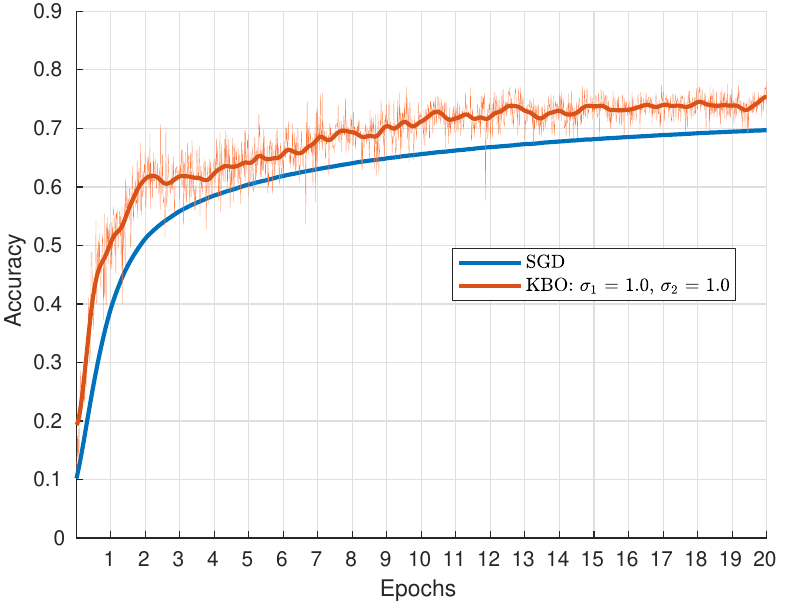}}\hfill\subfigure[KBO with particle reduction.\label{sfig:KOMsired}]{\includegraphics[width=0.5\textwidth]{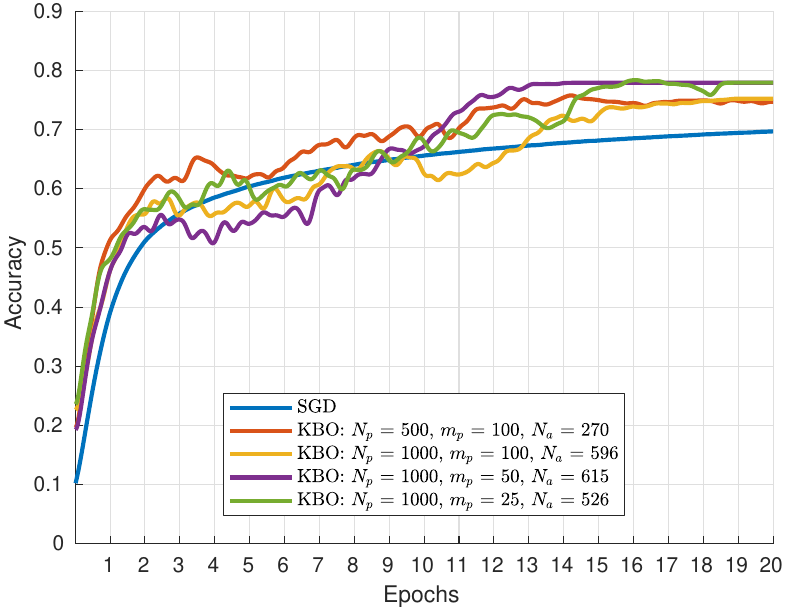}}
\end{center}
\caption{\AB{Performance comparison among SGD and KBO. The line referring to SGD shows the average over 500 simulations. The orange line refer to the KBO where both microscopic and macroscopic estimate are employed. The plot on the left depicts the performance of the KBO approach using $N_p=500$ without any particle reduction strategy (the solid line is a smooth representation of the shaded one), while the plot on the right refers to the adoption of \cref{eq:npartdec} with $\mu=0.1$ with different choices for particle numbers $N_p$ and particles' batch $m_p$. The average number of particles is denoted by $N_a$.}}
\label{fig:MNISTcomp}
\end{figure}

We run 500 simulations for SGD, since these runs are equivalent to one simulation of KBO with 500 particles. \cref{sfig:KOMnored} shows the accuracy obtained on the validation test all over the epochs. For computing the accuracy achieved by KBO, the parameters of the neural network are set as the macroscopic estimate reached at each iteration. The line referring to SGD corresponds to the average accuracy over the 500 simulations. \AB{In the numerical tests, the results obtained through the KBO method were shown to be superior in terms of accuracy to those obtained with classical SGD. A further test shows how the diminishing particle strategy depicted in \cref{eq:npartdec} is very effective even in this context: starting with 500 particles and setting $\mu=0.1$ ends the entire computation with just 270 particles, having a remarkable speed up in terms of computational time (see \cref{sfig:KOMsired}). Beside the diminishing strategy, several coupling of number of particles and batch size have been tested in \cref{sfig:KOMsired}: all of these setting lead to reliable results. Moreover, as already observed in \cref{subsec:compSGD}, SGD is quite sensitive to the starting point, whereas KBO is able to reach similar performances with different initializations as shown in \cref{fig:MNISTcomp2}.} 

\begin{figure}[htb]
\begin{center}
\includegraphics[width=0.5\textwidth]{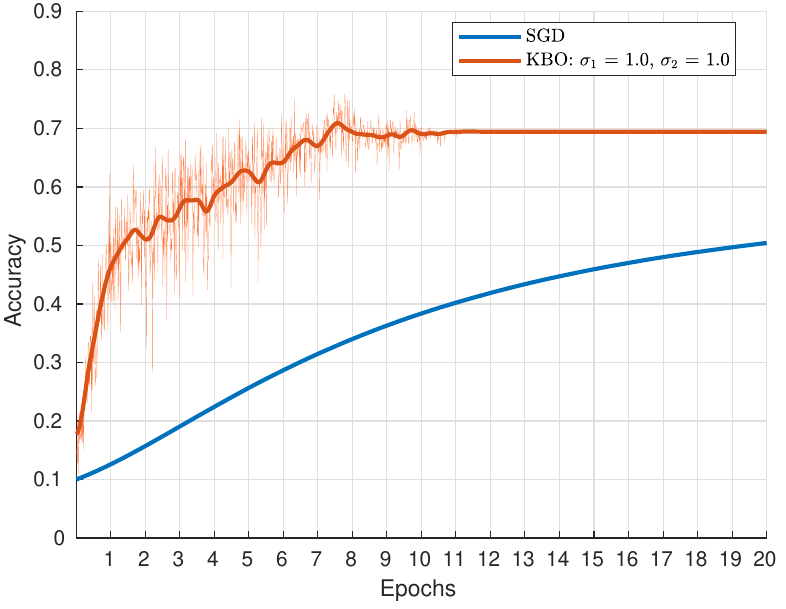}
\end{center}
\caption{\AB{Comparison of SGD and KBO performances when the starting point and the particles are randomly chosen as realizations of a Gaussian distribution of zero mean and standard deviation equal to 10. KBO is set to employ the strategy depicted in \cref{eq:npartdec} with $\mu=0.1$. The initial number of particles is $N_p=500$.}}
\label{fig:MNISTcomp2}
\end{figure}

\section{Conclusions}
%\section{Conclusions}
In this work we have presented a new gradient free method based on a kinetic dynamics characterized by binary interactions between particles. Unlike previously introduced consensus-based optimization (CBO) methods, the binary interaction process in the limit of a large number of particles does not correspond to a mean-field dynamics but to a Boltzmann-type dynamics inspired by classical kinetic theory. To our knowledge these are the first metaheuristic algorithms based on a Boltzmann-like dynamics for the identification of the global minimum. Compared to CBO methods, the kinetic theory based optimization method (KBO) introduced here can be seen as a mathematical formalism related to the use of mini-batches of interacting particles of size $2$. The KBO method, uses both local binary information and global information to explore the search space. \rev{In both cases, we have been able to prove convergence to the global minimum under reasonable assumptions on the objective function using techniques inspired} by those introduced in \cite{carrillo2019consensus}. 

The numerical experiments reported have demonstrated the excellent performance of the KBO technique both in the case of high dimensional problems with benchmark test functions, and in the case of applications to machine learning.  It is remarkable that the method can achieve good success rates also in the cases where no global information is used in the dynamics, namely there is only limited communication restricted to particles interacting by pairs. In this case, convergence to the global minimum can be seen as an emerging phenomena of a very simple dynamic where particles are not forced to converge towards a collective estimate of the global minimum. 

\rev{On the other hand, from a mathematical viewpoint, the case with only local information is more difficult and convergence to global minimum requires more restrictive conditions on the parameters. These restrictions, however, become less stringent as soon as the method is used in combination with global information.} In the sequel we plan to address our attention more specifically to the analysis of the Monte Carlo algorithms used in the KBO implementation and to the possible extension of the present methodology to non homogeneous dynamics in the spirit of particle swarm optimization as in \cite{gp20}.

\section*{Acknowledgements} 
This work has been written within the
activities of GNCS groups of INdAM (National Institute of
High Mathematics). The support of MIUR-PRIN Project 2017, No. 2017KKJP4X “Innovative numerical methods for evolutionary partial differential equations and applications” is acknowledged. The work of G. Borghi is funded by the Deutsche Forschungsgemeinschaft (DFG, German Research Foundation) – Projektnummer 320021702/GRK2326 – Energy, Entropy, and Dissipative Dynamics (EDDy).

\appendix

\section{Test functions for global optimization}
In the sequel we report the test function for global optimization used in the numerical examples. For more detailed information, see \cite{JamilY13}.
\begin{itemize}
\item Sphere
$$
f(x) = \sum_{i=1}^{d} (x_i-b_i)^2
$$
where $b\in\RR^{d}$ is a random vector belonging to the hypercube $[-5,5]^{d}$. The minimum is achieved in $x^\star=0$ and $f(x^\star)=0$.
\item Styblinski-Tank function
$$
f(x) = \frac{1}{2}\sum_{i=1}^{d} (x_i^4 -16x_i^2+5x_i)
$$
whose minimizer is $x^\star=(-2.903534,\ldots,-2.903534)$ and $f(x^\star)=-39.16599d$. The function is evaluated in $[-5,5]^d$.
\item Ackley Function.
$$
f(x) =  -20\exp\lp -0.2\sqrt{\frac{1}{d}\sum_{i=1}^{d}x_i^2}\rp-\exp\lp\frac{1}{d}\sum_{i=1}^{d}\cos(2\pi x_i)\rp+ 20 +e
$$
whose sole minimizer is $x^\star=0$ and $f(x^\star)=0$. The function is evaluated in $[-32,32]^d$. 
%\item Xing--She--Yang function, type I.
%$$
%f(x)=\sum_{i=1}^{20}\xi_i|x_i|^i
%$$
%where $\xi_i$, $i=1,\ldots,20$, is a realization of uniform random variable in $[0,1]$. The minimizer is again $x^\star=0$ and $f(x^\star)=0$.

\item Grienwank Function.
$$
f(x) = 1 + \sum_{i=1}^{d} \frac{x_i^{2}}{4000} - \prod_{i=1}^{d}\cos\lp\frac{x_i}{\sqrt{i}}\rp
$$
and $x^\star=0$, $f(x^\star)=0$. The function is evaluated in $[-600,600]^d$.

\item Negative Exponential function.
$$
f(x) = - \exp\lp-\frac12 \sum_{i=1}^d (x_i-b_i)^2\rp
$$
with $b\in\RR^d$. $x^\star = b$, $f(x^\star)=-1$, the function is evaluated in $[-5,5]^d$. 

\item Rastrigin function
$$
\label{eq:rastrigin}
f(x) = \frac1d\sum_{i=1}^d \lp x_i^2 -10\cos(2\pi x_i)\rp +10
$$
and $x^\star=0$, $f(x^\star)=0$. The function is evaluated in $[-5.12,5.12]^d$.

%\item Schwefel 2.20 Function.
%$$
%f(x) = \sum_{i=1}^{20}|x_i|
%$$
%which reads actually as the $\ell_1$ norm of its argument. Obviously, its sole minimizer is $x^\star=0$ and $f(x^\star)=0$. The function is evaluated in $[-5,5]$. 

%\item Schwefel 2.21 Function.
%$$
%f(x) = \max_{i=1,\ldots,20}|x_i|
%$$
%whose sole minimizer is $x^\star=0$ and $f(x^\star)=0$. The function is evaluated in $[-100,100]^d$. 

\item Schwefel 2.22 Function.
$$
f(x)=\sum_{i=1}^{d}|x_i|+\prod_{i=1}^{d}|x_i|
$$
its sole minimizer is $x^\star=0$ and $f(x^\star)=0$. The function is evaluated in $[-100,100]^d$.
\item Schwefel 2.23 Function.
$$
f(x) = \sum_{i=1}^{d}x_i^{10}
$$
whose minimizer is $x^\star=0$ and $f(x^\star)=0$. The function is evaluated in $[-100,100]^d$. 

\item Salomon function
$$
f(x) = 1-\cos\lp2\pi\sqrt{\sum_{i=1}^{d}x_i^2}\rp+0.1\sqrt{\sum_{i=1}^{d}x_i^2}
$$
with $x^\star= 0$ and $f(x^\star)=0$. The evaluation of this function is done in $[-100,100]^d$. 

\item Sum of squares
$$
f(x) = \sum_{i=1}^d ix_i^2
$$
whose sole minimizer is again the origin and the value in the minimizer is $0$. It is evaluated in $[-10,10]^d$.
%
%\item Xin--She Yang function
%$$
%f(x) = \sum_{i=1}^d\xi_i|x_1|^i
%$$
%where $\xi_i$ is drawn from a uniform distribution in $[0,1]$. The minimizer is $x^\star=0$ and $f(x^\star)=0$. The function is evaluated in $[-5,5]^d$.
\end{itemize}

\bibliographystyle{abbrv}
\bibliography{bibfile}
%%%%%%%%%%%%%%%%%%%%%%%%
%\begin{thebibliography}{99}
%\end{thebibliography}
%%%%%%%%%%%%%%%%%%%%%%%%

\end{document}